\documentclass{amsart}
\usepackage{amsmath}
\usepackage{amssymb}
\usepackage{array}
\usepackage{amscd}
\usepackage{tikz-cd}

\numberwithin{equation}{section}
\newtheorem{Th}[subsection]{Theorem}
\newtheorem*{Th*}{Theorem}
\newtheorem{Lemma}[subsection]{Lemma}
\newtheorem{Prop}[subsection]{Proposition}

\newtheorem{Cor}[subsection]{Corollary}
\theoremstyle{definition}
\newtheorem{definition}[subsection]{Definition}
\newtheorem*{definition*}{Definition}
\newtheorem{Remark}[subsection]{Remark}
\newtheorem{Example}[subsection]{Example}
\newcommand{\comm}[1]{}

\usepackage{color}
\usepackage[normalem]{ulem}
\definecolor{DarkGreen}{rgb}{0,0.5,0.1} 

\newcommand\soutD{\bgroup\markoverwith
{\textcolor{DarkGreen}{\rule[.5ex]{2pt}{1pt}}}\ULon}

\makeatletter
\newcommand*{\rom}[1]{\expandafter\@slowromancap\romannumeral #1@}
\makeatother

\begin{document}
 
\title[Birational automorphism groups of Severi--Brauer surfaces over $\mathbb{Q}$]{Birational automorphism groups of Severi--Brauer surfaces over the field of rational numbers}
\author{Anastasia V.~Vikulova}
\address{{\sloppy
\parbox{0.9\textwidth}{
Steklov Mathematical Institute of Russian
Academy of Sciences,
8 Gubkin str., Moscow, 119991, Russia
}\bigskip}}
\email{vikulovaav@gmail.com}
\date{}
\maketitle

\begin{abstract}
We prove that the only non-trivial finite subgroups of birational automorphism group of non-trivial Severi--Brauer surfaces over the field of rational numbers are~$\mathbb{Z}/3\mathbb{Z}$ and $(\mathbb{Z}/3\mathbb{Z})^2.$ Moreover,  we show that $(\mathbb{Z}/3\mathbb{Z})^2$  is contained in $\mathrm{Bir}(V)$ for any Severi--Brauer surface $V$ over a field of characteristic different from $2$ and $3$, and $(\mathbb{Z}/3\mathbb{Z})^3$ is contained in $\mathrm{Bir}(V)$ for any Severi--Brauer surface~$V$ over a field of characteristic different from $2$ and $3$ which contains a non-trivial cube root of unity. 
\end{abstract}

	\section{Introduction}

The Cremona group $\mathrm{Cr}_n(\mathbf{F})$  is a group of birational automorphisms of $\mathbb{P}^n$ over a field $\mathbf{F}.$ It is difficult to describe this group, except the case $n=1,$ when we have~\mbox{$\mathrm{Cr}_1(\mathbf{F}) \simeq \mathrm{PGL}_2(\mathbf{F}).$} Even the classification of finite subgroups seems extremely hard. Nowadays, we know the description of conjugacy classes of finite subgroups only for $\mathrm{Cr}_2(\mathbb{C})$  (see~\cite{DI}).

It is natural to ask how birational automorphisms of forms of projective spaces behave. 

\begin{definition}
An $n$-dimensional variety $V$ over a field $\mathbf{F}$ is called a \textit{ Severi--Brauer variety} if 
$$
V \times_{\mathrm{Spec}(\mathbf{F})}\mathrm{Spec}(\overline{\mathbf{F}}) \simeq \mathbb{P}^n_{\overline{\mathbf{F}}},
$$

\noindent where $\overline{\mathbf{F}}$ is an algebraic closure of $\mathbf{F}.$ Such a variety $V$ is called non-trivial if it is not isomorphic to $\mathbb{P}^n_{\mathbf{F}}.$
\end{definition}

Like $\mathrm{Cr}_n(\mathbf{F}),$ the group $\mathrm{Bir}(V)$ of  birational automorphisms of a Severi--Brauer variety $V$ also has a complicated structure (cf.~\cite{IskTregub} and~\cite{Weinstein}). A classification of finite groups that appear as subgroups of $\mathrm{Bir}(V)$ for non-trivial Severi--Brauer surfaces over various fields of characteristic zero was given in~\cite[Theorem 1.3]{Shramov1}. On the other hand, there is no simple way to decide which of them are realized for a given field, or for a given Severi--Brauer surface.

\begin{Th}[{\cite[Corollary 1.4]{Shramov1}}]\label{ThofShramov}
Let $V$ be a non-trivial Severi--Brauer surface over $\mathbb{Q}$   and let $G \subset \mathrm{Bir}(V)$ be a finite subgroup. Then we have~\mbox{$G \subset (\mathbb{Z}/3\mathbb{Z})^3.$}
\end{Th}

The goal of this paper is to prove the following result which is a strengthening of Theorem~\ref{ThofShramov}.

\begin{Th}\label{Z/3Z}
Let $V$ be a non-trivial Severi--Brauer surface  over the field $\mathbb{Q}$ and let $G$ be a finite group. Then $G$ is isomorphic to a subgroup of $\mathrm{Bir}(V)$ if and only if~\mbox{$G \subset (\mathbb{Z}/3\mathbb{Z})^2,$} and $G$ is isomorphic to a subgroup of $\mathrm{Aut}(V)$ if and only if~\mbox{$G \subset \mathbb{Z}/3\mathbb{Z}.$}
\end{Th}

Recall the following result of A.Beauville which is a particular case of~\cite[Theorem and (3.2)]{Beauville}.

\begin{Th}\label{ThofBeauville}
Let $\mathbf{F}$ be an algebraically closed field of characteristic different from~$3.$ Then we have 

\begin{enumerate}
\renewcommand\labelenumi{\rm (\roman{enumi})}
\renewcommand\theenumi{\rm (\roman{enumi})}

\item\label{ThofBeauvillebir3}
$\mathrm{Bir}(\mathbb{P}^2_{\mathbf{F}}) \supset (\mathbb{Z}/3\mathbb{Z})^3;$

\item\label{ThofBeauvillebir4}
$
\mathrm{Bir}(\mathbb{P}^2_{\mathbf{F}}) \not\supset (\mathbb{Z}/3\mathbb{Z})^4;
$
\item\label{ThofBeauvilleaut3}
$
\mathrm{Aut}(\mathbb{P}^2_{\mathbf{F}}) \not\supset (\mathbb{Z}/3\mathbb{Z})^3.
$

\end{enumerate}

\end{Th}

 In the process of proving Theorem~\ref{Z/3Z} we obtain the following result which can be considered as an analogue of Theorem \ref{ThofBeauville} for arbitrary Severi--Brauer surfaces.

\begin{Prop}\label{propZ/3Z}
Let $V$ be a Severi--Brauer surface over a perfect field $\mathbf{F}$ of characteristic different from $2$ and $3.$ Then

\begin{enumerate}
\renewcommand\labelenumi{\rm (\roman{enumi})}
\renewcommand\theenumi{\rm (\roman{enumi})}

\item\label{eq:birsubset}
$\mathrm{Bir}(V) \supset (\mathbb{Z}/3\mathbb{Z})^2;$

\item\label{eq:birnotsubset}
$\mathrm{Bir}(V) \supset (\mathbb{Z}/3\mathbb{Z})^3$ if and only if   $\mathbf{F}$ contains a non-trivial cube root of unity;

\item\label{eq:birnotsubset4}
$\mathrm{Bir}(V) \not\supset (\mathbb{Z}/3\mathbb{Z})^4;$ 

\item\label{eq:autsubset} 
$\mathrm{Aut}(V) \supset \mathbb{Z}/3\mathbb{Z};$

\item\label{eq:autnotsubset} 
$\mathrm{Aut}(V) \supset (\mathbb{Z}/3\mathbb{Z})^2$ if and only if $\mathbf{F}$ contains a non-trivial cube root of unity;

\item\label{eq:autnotsubset3} 
$\mathrm{Aut}(V) \not\supset (\mathbb{Z}/3\mathbb{Z})^3.$ 

\end{enumerate}

\end{Prop}

\begin{Remark}
The existence of birational actions of the group $(\mathbb{Z}/3\mathbb{Z})^3$ on certain non-trivial Severi--Brauer surfaces was proved in~\cite[Theorem 1.2]{Shramov3} (see also \mbox{\cite[Section~3]{Shramov3}} for construction of an example of such an action). Proposition~\ref{propZ/3Z}\ref{eq:birnotsubset} strengthens this result by showing that such an action exists on \emph{every} Severi--Brauer surface over a field of characteristic different from $2$ and $3$ containing a non-trivial cube root of unity.
\end{Remark}

Let us briefly explain the idea of the proof of Proposition \ref{propZ/3Z}\ref{eq:birsubset} and \ref{eq:birnotsubset} and, as a result, of Theorem~\ref{Z/3Z} immediately follows from it. We show that~$\mathbb{Z}/3\mathbb{Z}$ acts biregularly on $V.$ But~$(\mathbb{Z}/3\mathbb{Z})^2$ does not if $\mathbf{F}$ does not contain a non-trivial cube root of unity. However, the group $(\mathbb{Z}/3\mathbb{Z})^2$ acts birationally on every Severi--Brauer surface.   To this end we blow up the Severi--Brauer surface $\mathbb{Z}/3\mathbb{Z}$-equivariantly, obtain a smooth cubic surface and observe that it is isomorphic to the Fermat cubic surface over an algebraic closure of $\mathbf{F}$. Studying all $3$-subgroups in the automorphism group of the Fermat cubic surface, which commute with the Galois group $\mathrm{Gal}(\overline{\mathbf{F}}/\mathbf{F})$  we get that~$(\mathbb{Z}/3\mathbb{Z})^2$ acts biregularly on this cubic surface. The remaining assertions in Proposition \ref{propZ/3Z} are easy.

The plan of the paper is as follows. 
In Section~\ref{sec:preliminaries} we prove some supplementary lemmas.
In Section~\ref{section:SV} we collect some basic facts about Severi--Brauer surfaces and study finite subgroups of their automorphism groups.
In Section~\ref{sectionofcubic} we collect some auxiliary facts about cubic surfaces and construct a birational action of $(\mathbb{Z}/3\mathbb{Z})^2$ on Severi--Brauer surfaces.
In Section~\ref{section:proofresults} we study $3$-groups in the birational automorphism groups of Severi--Brauer surfaces and we prove Proposition~\ref{propZ/3Z}.
In Section~\ref{section:proofoftheorem} we prove Theorem~\ref{Z/3Z}.

\textbf{Notation.} We assume that all fields in the paper are perfect. Let $X$ be a variety defined over $\mathbf{F}.$ If $\mathbf{F} \subset \mathbf{L}$ is an extension of $\mathbf{F},$ then we will denote by $X_{\mathbf{L}}$ the variety
$$
X_{\mathbf{L}}=X \times_{\mathrm{Spec}(\mathbf{F})}\mathrm{Spec}(\mathbf{L}).
$$

\noindent By $\overline{\mathbf{F}}$ we denote an algebraic closure of $\mathbf{F}.$ A geometric point of $X$ is a point of $X_{\overline{\mathbf{F}}}.$  A geometric line on $X$ is a line on $X_{\overline{\mathbf{F}}}.$

\textbf{Acknowledgment.} The author would like to express the deepest gratitude to C.~A.~Shramov for suggesting this problem, for constant attention to this paper and many hints. Also the author is extremely grateful to A.~S.~Trepalin for numerous discussions about cubic surfaces, and to A.~O.~Savelyeva for conversations  about   Severi--Brauer varieties.  Finally, the author thanks referees for deep reading of this paper and extremely important remarks.

This work was performed at the Steklov International Mathematical Center and supported by the Ministry of Science and Higher Education of the Russian Federation (agreement no. 075-15-2022-265). The author was partially supported  the Moebius Contest Foundation for Young Scientists and  by the BASIS foundation grant "Young Russian Mathematics".

\section{Preliminaries}\label{sec:preliminaries}

In this section we collect some auxiliary facts.

\begin{Lemma}\label{lemma:representationz/3z}
Let $\mathbf{F}$ be a field of  characteristic different from $3$ which does not contain non-trivial cube roots of unity. Then 

\begin{enumerate}
\renewcommand\labelenumi{\rm (\roman{enumi})}
\renewcommand\theenumi{\rm (\roman{enumi})}

\item\label{lemma:representationz/3z23}
$(\mathbb{Z}/3\mathbb{Z})^2 \not\subset \mathrm{GL}_3(\mathbf{F});$

\item\label{lemma:representationz/3z34}
$(\mathbb{Z}/3\mathbb{Z})^3 \not\subset \mathrm{GL}_4(\mathbf{F}).$
\end{enumerate}

\end{Lemma}

\begin{proof}
Let us prove~\ref{lemma:representationz/3z23}. Assume that there is a linear action of $(\mathbb{Z}/3\mathbb{Z})^2$ on the vector space $\mathbf{F}^3.$ As this group is abelian we get that the matrices, which represent the elements of the group, are simultaneously diagonalizable  over $\overline{\mathbf{F}}$. Thus, the elements of the group are conjugate to
\begin{equation}\label{eq:matricesomegaabcdwitout323}
\begin{pmatrix}
\omega^a & 0 & 0\\
0 & \omega^b & 0\\
0 & 0 & \omega^c 
\end{pmatrix},
\end{equation}

\noindent where $\omega$ is a non-trivial cube root of unity and $a,b,c \in \{0,1,2\}.$ Note that the determinant and the trace of these matrices belong to $\mathbf{F}.$ Therefore, such matrices have to satisfy the following:
\begin{gather}
a+b+c \equiv 0 \mod 3; \label{eq:123} \\
\omega^a+\omega^b+\omega^c \in \mathbf{F}. \label{eq:223}
\end{gather}
\noindent  The condition~\eqref{eq:123} gives us only $9$ matrices. It is not hard to see that there are matrices of the form \eqref{eq:matricesomegaabcdwitout323} which satisfy \eqref{eq:123}, but do not satisfy \eqref{eq:223}. For example, 
$$
a=b=c=1.
$$
 This means that the order of the group consisting of the matrices \eqref{eq:matricesomegaabcdwitout323}, which satisfy \eqref{eq:123} and \eqref{eq:223},  is less then $9.$ This completes the proof.

Assertion~\ref{lemma:representationz/3z34} is proved in the same way. Indeed, assume that there is a linear action of $(\mathbb{Z}/3\mathbb{Z})^3$ on the vector space $\mathbf{F}^4.$ As this group is abelian we get that the matrices, which represent the elements of the group, are simultaneously diagonalizable  over $\overline{\mathbf{F}}$. Thus, the elements of the group are conjugate to
\begin{equation}\label{eq:matricesomegaabcdwitout3}
\begin{pmatrix}
\omega^a & 0 & 0 & 0\\
0 & \omega^b & 0 & 0\\
0 & 0 & \omega^c & 0\\
0 & 0 & 0 & \omega^d
\end{pmatrix},
\end{equation}

\noindent where $\omega$ is a non-trivial cube root of unity and $a,b,c,d \in \{0,1,2\}.$ Note that the determinant and the trace of these matrices belong to $\mathbf{F}.$ Therefore, such matrices have to satisfy the following:
\begin{gather}
a+b+c+d \equiv 0 \mod 3; \label{eq:1} \\
\omega^a+\omega^b+\omega^c+\omega^d \in \mathbf{F}. \label{eq:2}
\end{gather}
\noindent  The condition (\ref{eq:1}) gives us only $27$ matrices. It is not hard to see that there are matrices of the form \eqref{eq:matricesomegaabcdwitout3} which satisfy \eqref{eq:1}, but do not satisfy \eqref{eq:2}. For example, 
$$
a=b=c=1 \;\; \text{and} \;\; d=0.
$$
\noindent  This means that the order of the group consisting of the matrices \eqref{eq:matricesomegaabcdwitout3}, which satisfy \eqref{eq:1} and \eqref{eq:2},  is less then $27.$ This completes the proof.

\end{proof}

Let $X \subset \mathbb{P}^n$ be a projective variety over a field $\mathbf{F}.$ Denote by $\mathrm{Ir}(X_{\overline{\mathbf{F}}})$ the set of irreducible components of~$X_{\overline{\mathbf{F}}}.$ For an element~$\phi \in \mathrm{PGL}_{n+1}(\overline{\mathbf{F}})$ let 
$$
\phi_{\mathrm{Ir}}:\mathrm{Ir}(X_{\overline{\mathbf{F}}}) \to \mathrm{Ir}(\phi(X_{\overline{\mathbf{F}}}))
$$  
\noindent be the induced map between the sets  of irreducible components.

\begin{Lemma}\label{lemma:irreducibleGalois}
Let $\mathbf{F}$ be an arbitrary field. Let $X \subset \mathbb{P}^n$ be a projective variety defined over $\mathbf{F}$. Let $\phi \in \mathrm{PGL}_{n+1}(\overline{\mathbf{F}})$ be an element such that~\mbox{$X'=\phi(X)$} is defined over~$\mathbf{F}.$ Assume also that $\phi_{\mathrm{Ir}}$ commutes with the action of the Galois group~$\mathrm{Gal}(\overline{\mathbf{F}}/\mathbf{F})$ on~$\mathrm{Ir}(X_{\overline{\mathbf{F}}})$ and $\mathrm{Ir}(X'_{\overline{\mathbf{F}}}).$  Assume that any irreducible component of~$X_{\overline{\mathbf{F}}}$ is a linear subspace in $\mathbb{P}^n_{\overline{\mathbf{F}}}.$ Then there is $\psi \in \mathrm{PGL}_{n+1}(\mathbf{F})$ such that $\psi(X)=X'$ and $\psi_{\mathrm{Ir}}=\phi_{\mathrm{Ir}}.$

\end{Lemma}

\begin{proof}
 Both $X_{\overline{\mathbf{F}}}$ and $X'_{\overline{\mathbf{F}}}$ decompose into unions of linear subspaces over $\overline{\mathbf{F}},$ i.e.
$$
X_{\overline{\mathbf{F}}}=\mathcal{L}_1 \cup \ldots \cup \mathcal{L}_m, \quad X'_{\overline{\mathbf{F}}}=\mathcal{L}'_1 \cup \ldots \cup \mathcal{L}'_m.
$$
\noindent Since $\phi_{\mathrm{Ir}}$ commutes with the Galois group, we can reorder $\mathcal{L}_i$ and $\mathcal{L}'_i$ so that 
$$
\phi(\mathcal{L}_i)=\mathcal{L}'_{i}.
$$
\noindent  For $g \in \mathrm{Gal}(\overline{\mathbf{F}}/\mathbf{F})$ and $1 \leqslant i \leqslant m,$ write $g(i) \in \{1, \ldots, m\}$ such that 
$$
g(\mathcal{L}_i)=\mathcal{L}_{g(i)}.
$$
\noindent  Then we have $ g(\mathcal{L}'_i)=\mathcal{L}'_{g(i)}.$

 For all $1 \leqslant i \leqslant m$  denote by~$\mathcal{F}_i$ the set of linear homogeneous polynomials 
$$
f \in  \overline{\mathbf{F}}[x_0, \ldots, x_n]
$$
\noindent such that~$f\vert_{\mathcal{L}'_i} = 0.$ Denote by 
$$
\mathcal{M} \subset \mathbb{P} \bigl( \mathrm{Mat}_{n+1}(\overline{\mathbf{F}})\bigr)  \simeq  \mathbb{P}^{(n+1)^2-1}_{\overline{\mathbf{F}}}
$$

\noindent  the set of all non-degenerate matrices $\psi$ such that $\psi(\mathcal{L}_i)=\mathcal{L}'_i$ for all $1 \leqslant i \leqslant m.$

 For every~\mbox{$i,$~$P \in \mathcal{L}_i$} and $f \in \mathcal{F}_i$ denote by $R^i_{Pf}$ the linear relation 
$$
f(\psi(P))=0
$$
\noindent on the entries of the matrix $\psi.$ Let $\mathcal{R}$ be the set of all such $R^i_{Pf}$ for all 
$$
i \in \{1, \ldots, m\}, \quad P \in \mathcal{L}_i, \quad f \in \mathcal{F}_i.
$$
\noindent So $\mathcal{M}$ is an intersection of $\mathrm{PGL}_{n+1}(\overline{\mathbf{F}}) \subset \mathbb{P}^{(n+1)^2-1}_{\overline{\mathbf{F}}}$ with a linear subspace $\widetilde{\mathcal{M}}$ which is defined by the equations~$R^{i}_{Pf}.$ Note that $\widetilde{\mathcal{M}}$ is non-empty since $\phi \in \widetilde{\mathcal{M}}.$   Take an element $g$ of the Galois group. We obtain
$$
g(R^i_{Pf})=R^{g(i)}_{g(P)g(f)}.
$$

\noindent Let us prove that $R^{g(i)}_{g(P)g(f)} \in \mathcal{R}.$ Indeed, we have
$$
g(\mathcal{L}_i)=\mathcal{L}_{g(i)}, \quad g(P) \in \mathcal{L}_{g(i)} \quad \text{and} \quad g(f)\vert_{\mathcal{L}'_{g(i)}}=0,
$$
\noindent where the last equality holds because $\phi_{\mathrm{Ir}}$ commutes with the Galois group.  Therefore,~$\widetilde{\mathcal{M}}$ is Galois-invariant. Thus, $\widetilde{\mathcal{M}}$ is defined over $\mathbf{F}.$

Assume that $\mathbf{F}$ is an infinite field. Then there is a dense set of~$\mathbf{F}$-points on~$\widetilde{\mathcal{M}}.$ The intersection~$\mathcal{M}=\widetilde{\mathcal{M}} \cap \mathrm{PGL}_{n+1}(\overline{\mathbf{F}})$ is non-empty because it contains~$\phi.$ Therefore, $\mathcal{M}$ contains an $\mathbf{F}$-point, and so there exists $\psi \in  \mathrm{PGL}_{n+1}(\mathbf{F})$ as in the statement.

Assume that $\mathbf{F}$ is a finite field.  Consider the algebraic group
$$
\mathcal{T}=\{\tau \in \mathrm{PGL}_{n+1}(\overline{\mathbf{F}}) \mid \tau(\mathcal{L}_i)=\mathcal{L}_i \;\; \text{for all} \;\; 1 \leqslant i \leqslant m\}.
$$
Since the Galois group permutes $\mathcal{L}_i,$ the group~$\mathcal{T}$ is~$\mathrm{Gal}(\overline{\mathbf{F}}/\mathbf{F})$-invariant, thus it is defines over~$\mathbf{F}.$

We get $\mathcal{M}=\phi\mathcal{T},$ since we have $\phi(\tau(\mathcal{L}_i))=\mathcal{L}'_i$ for all $1 \leqslant i \leqslant m$ and any $\tau \in \mathcal{T},$ and~\mbox{$\phi^{-1}(\mu(\mathcal{L}_i))=\mathcal{L}_i$} for all $1 \leqslant i \leqslant m$ and any $\mu \in \mathcal{M}.$ The group $\mathcal{T}$ acts on $\mathcal{M}$ on the right by composition, that is we have the morphism
\begin{gather*}
T \colon \mathcal{M} \times \mathcal{T} \to \mathcal{M}\\
(\mu, \tau) \mapsto \mu\tau.
\end{gather*}
\noindent This morphism is defined over~$\mathbf{F}.$ Moreover,  we have an isomorphism 
\begin{equation}\label{eq:isomtorsor}
F \colon \mathcal{T}_{\overline{\mathbf{F}}} \simeq \mathcal{M}_{\overline{\mathbf{F}}},
\end{equation}
 defined as $F(\tau)=\phi\tau$ and the action of $\mathcal{T}_{\overline{\mathbf{F}}}$ on $\mathcal{M}_{\overline{\mathbf{F}}}$ is free and transitive.  This means that~$\mathcal{M}$ is a $\mathcal{T}$-torsor (see, for example,~\cite[Section~2.6]{ShramovGorchinskiy}). Since~$\mathcal{M}$ is a non-empty open subset of the linear subspace~$\widetilde{\mathcal{M}},$ we get that it is smooth and geometrically connected variety, so by the isomorphism~\eqref{eq:isomtorsor} we get that $\mathcal{T}$ is a smooth connected group variety. So we can use~\cite[Corollary~17.98]{Milnegroups} and get that~$\mathcal{M}$ is a trivial $\mathcal{T}$-torsor. This means that $\mathcal{M}$ has an $\mathbf{F}$-point, and so there exists $\psi \in  \mathrm{PGL}_{n+1}(\mathbf{F})$ as in the statement.

\end{proof}

Now let us conclude the section by a simple lemma about field extensions.

\begin{Lemma}\label{lemma:fielddegree3}
Let $\mathbf{F}$ be a field of characteristic different from $3.$ Let us consider the extension $\mathbf{L}$ of $\mathbf{F}$ of degree $3$ generated by an element $\bar{x} \in \mathbf{L}$ such that $\bar{x}^3-\alpha=0$ for $\alpha \in \mathbf{F}^*.$ Let us consider an element $\bar{y} \in \mathbf{L}$ such that $\bar{y}^3-\beta=0$ for $\beta \in \mathbf{F}^*.$  Then either $\bar{y}=c,$ or $\bar{y}=c\bar{x},$ or $\bar{y}=c\bar{x}^2$ for $c \in \mathbf{F}^*.$
\end{Lemma}

\begin{proof}
First of all, assume that $\beta \in \left(\mathbf{F}^*\right)^3.$ Then  $\bar{y}$ is an element in  $\mathbf{F}.$ Therefore, we can assume that $\beta \notin  \left(\mathbf{F}^*\right)^3,$ thus the element $\bar{y}$ generates a field extension of degree $3.$ This means that $1$, $\bar{y}$ and $\bar{y}^2$ is a basis of $\mathbf{L}.$

We have
$$
\bar{y}=u+v\bar{x}+w\bar{x}^2
$$
\noindent for $u,v,w \in \mathbf{F}.$ Let us consider the trace of the multiplication by the element $\bar{y}.$  Note that the trace of both elements $\bar{x}$ and~$\bar{x}^2$ is zero, which can be easily computed in the basis $1$, $\bar{x}$ and $\bar{x}^2.$ Therefore, on the one hand, the trace of multiplication by $\bar{y}$ is equal to $3u.$ On the other hand, if we consider the multiplication by the element $\bar{y}$ in the basis $1$, $\bar{y}$ and $\bar{y}^2,$ we get that its trace is equal to zero. So we obtain  $\bar{y}=v\bar{x}+w\bar{x}^2.$ By assumption we get
$$
\beta=\bar{y}^3=v^3\alpha+w^3\alpha^2+3\bar{x}\alpha v^2w+3\bar{x}^2\alpha vw^2.
$$
\noindent So we obtain $v^2w=vw^2=0$ which implies that either $v=0,$ or $w=0,$  and we are done.

\end{proof}

\section{Severi--Brauer surfaces}\label{section:SV}
In this section we study finite subgroups of the automorphism groups of  Severi--Brauer surfaces. Let us mention some properties of Severi--Brauer varieties (for more details see, for instance,~\cite{SimpleAlgebras} and~\cite{Kollar}). If $V$ is a Severi--Brauer variety over a field~$\mathbf{F},$ then $V$ is non-trivial  if and only if~\mbox{$V(\mathbf{F})= \varnothing$} (see~\cite[Theorem 5.1.3]{SimpleAlgebras}). There is a bijection between Severi--Brauer varieties of dimension~$n$ over a field~$\mathbf{F}$ and central simple algebras of dimension~$(n+1)^2$ over  $\mathbf{F}$  (see e.g.~\cite[\S 5.2]{SimpleAlgebras}).

\begin{Th}[{\cite[Theorems 2.1.3 and 5.2.1]{SimpleAlgebras}}]\label{CSAcorps}
Let $V$ be a non-trivial Severi--Brauer variety of dimension $n$ such that $n+1$ is a prime number. Then the central simple algebra $A$ which is associated to $V$ is a division algebra.  
\end{Th}

The following theorem describes automorphism groups of Severi--Brauer varieties.

\begin{Th}[{see, for example,~\cite[p.~266]{Chatelet} or \cite[Lemma 4.1]{ShramovVologodsky}}]\label{csaprop}
If a central simple algebra $A$ corresponds to a~Severi--Brauer variety~$V$ over a field $\mathbf{F}$ then we have~\mbox{$\mathrm{Aut}(V)  \simeq A^*/\mathbf{F}^*.$}
\end{Th}

Theorem \ref{csaprop} allows to obtain restrictions on the orders of automorphisms of Severi--Brauer varieties.

\begin{Example}[{cf.~\cite[Lemma 5.2]{Shramov1}}]\label{exampleQF}
Let $\mathbf{F}$ be a field of characteristic different from~$3.$ Let~$V$ be a non-trivial Severi--Brauer surface and let $x \in \mathrm{Aut}(V)$ be an element of prime order $p \neq 3.$ We claim that $p \equiv 1 \pmod{3}.$ Indeed, let $A$ be a central simple algebra which corresponds to $V.$ Then $A$ is a division algebra by Theorem~\ref{CSAcorps}. By Theorem~\ref{csaprop} we have $\mathrm{Aut}(V)  \simeq A^*/\mathbf{F}^*.$ 

Let $\bar{x}$ be any element in the preimage of $x$ under the homomorphism~\mbox{$A^* \to  A^*/\mathbf{F}^*.$} Then we have~$\bar{x}^p=a \in \mathbf{F}^*.$ Let $B \subset A$ be the field  generated by $\bar{x}$ over $\mathbf{F}.$ Since
$$
\dim_{\mathbf{F}}A=\dim_{B}A \cdot \dim_{\mathbf{F}}B,
$$
\noindent we obtain $\dim_{\mathbf{F}}B=3,$ because $\dim_{\mathbf{F}}B=9$ is impossible as $B$ is a commutative algebra while $A$ is not. 

Let $f(t)$ be a minimal polynomial of $\bar{x}$ over $\mathbf{F}.$ Its degree is equal to $3.$ Moreover, the polynomial $f(t)$ divides $t^p-a.$ In particular, the polynomial $t^p-a$ is reducible over $\mathbf{F}.$ So by~\cite[Theorem \rom{6}.9.1]{Lang} we get~\mbox{$a=c^p$} for some $c \in \mathbf{F}^*. $ This means that the roots of $f(t)$ in $\overline{\mathbf{F}}$ are $c\xi_1$, $c\xi_2$ and $c\xi_3$, where~$\xi_i$ are pairwise different $p$-th roots of unity. Hence these $\xi_i$ for $1 \leqslant i \leqslant 3$ form a~$\mathrm{Gal}(\overline{\mathbf{F}}/\mathbf{F})$-orbit. Let  $\Gamma$ be the image of  $\mathrm{Gal}(\overline{\mathbf{F}}/\mathbf{F})$ in the automorphism group 
$$
\mathrm{Aut}(\mathbb{Z}/p\mathbb{Z}) \simeq \mathbb{Z}/(p-1)\mathbb{Z},
$$
\noindent where $\mathbb{Z}/p\mathbb{Z}$ is considered as the multiplicative group of  $p$-th roots of unity. Then the group $\Gamma$ has an orbit of order~$3.$ Therefore, the order of $\Gamma$ is divisible by $3.$ Thus, $3$ divides $p-1$ and we are done.

\end{Example}

The following lemma is a well-known fact about subgroups of the automorphism group of Severi--Brauer surfaces. We reproduce its proof for the convenience of the reader.

\begin{Lemma}[see e.g.~{\cite[Example 4.1 and Remark 4.2]{Shramov1}}]\label{lemma:autz/3z}
Let $V$ be a Severi--Brauer surface over a field $\mathbf{F}$ of characteristic different from $3.$ Then $\mathrm{Aut}(V)$  contains a subgroup isomorphic to~$\mathbb{Z}/3\mathbb{Z}.$ If $\mathbf{F}$ contains a non-trivial cube root of unity, then~$\mathrm{Aut}(V)$  contains a subgroup isomorphic to~$(\mathbb{Z}/3\mathbb{Z})^2.$
\end{Lemma}

\begin{proof}
If $V \simeq \mathbb{P}^2,$ then $\mathrm{Aut}(V) \simeq \mathrm{PGL}_3(\mathbf{F}),$ and there is a group of order $3$ in~$\mathrm{PGL}_3(\mathbf{F})$ which is generated by the element cyclically permuting the coordinates. If $\mathbf{F}$ contains a non-trivial cube root $\omega$ of unity  then the elements
$$
\alpha=
\begin{pmatrix}
1 & 0 & 0 \\
0  & \omega & 0\\
0 & 0 & \omega
\end{pmatrix}
\in \mathrm{PGL}_3(\mathbf{F}) \quad \text{and} \quad \beta=
\begin{pmatrix}
1 & 0 & 0 \\
0  & \omega & 0\\
0 & 0 & \omega^2
\end{pmatrix}
\in \mathrm{PGL}_3(\mathbf{F})
$$

\noindent generate a group $(\mathbb{Z}/3\mathbb{Z})^2 \subset \mathrm{PGL}_3(\mathbf{F}).$

 Now assume that $V$ is a non-trivial Severi--Brauer surface. Then the Severi--Brauer surface $V$ corresponds to a division algebra $A$ by Theorem~\ref{CSAcorps}, and this algebra is a cyclic algebra by~\cite[Chapter 7, Exercise 9]{SimpleAlgebras}. By~\cite[Proposition~2.5.2]{SimpleAlgebras} the algebra $A$ is generated by a Galois extension $\mathbf{F} \subset \mathbf{L}$ of degree $3$ and an element~\mbox{$\alpha \in A$} such that $\alpha \notin \mathbf{F}^*$ and~\mbox{$\alpha^3 \in \mathbf{F}^*.$} Furthermore, one has $\alpha\lambda=\sigma(\lambda)\alpha$ for all $\lambda \in \mathbf{L},$ where $\sigma$ is a generator of the Galois group of the extension. The element~$\alpha$ gives us an automorphism of order~$3.$ Therefore, one has $\mathbb{Z}/3\mathbb{Z} \subset \mathrm{Aut}(V).$
 
  Finally, if $\mathbf{F}$ contains a non-trivial cube root $\omega$ of unity then by Kummer theory (see, for example,~\cite[Chapter \rom{3}, \S 2, Lemma 2]{Frohlich}) we obtain $\mathbf{L}=\mathbf{F}(\beta)$ for some~\mbox{$\beta \notin \mathbf{F}$} such that  $\beta^3 \in \mathbf{F}^*$ and for a generator  $\sigma \in \mathrm{Gal}(\mathbf{L}/\mathbf{F})$ one has $\sigma(\beta)=\omega\beta.$ We have the following relations
$$
\alpha^3 \in \mathbf{F}^*, \; \beta^3 \in \mathbf{F}^*, \; \alpha\beta=\omega \beta\alpha.
$$  
\noindent  Therefore, the image of $\alpha$ and $\beta$ under the homomorphism $A^* \to A^*/\mathbf{F}^*$ generate the group $(\mathbb{Z}/3\mathbb{Z})^2 \subset A^*/\mathbf{F}^* \simeq \mathrm{Aut}(V).$

\end{proof}

Now we are going to study $3$-subgroups in the automorphism groups of Severi--Brauer surfaces. First of all, let us prove the following lemmas about central simple algebras of dimension $9.$

\begin{Lemma}\label{lemma:matrixalgebraifscalar}
Let $A$ be a central simple algebra of dimension~$9$ over a field $\mathbf{F}$ of characteristic different from $3.$ Consider the element $x \in A$ such that $x^3=1$ and~\mbox{$x \notin \mathbf{F}.$} Then $A$ is a matrix algebra. 

\end{Lemma}

\begin{proof}
Assume that $A$ is not a matrix algebra. Then by Wedderburn theorem (see, for example,~\cite[Theorem 2.1.3]{SimpleAlgebras}) we get that $A$ is a division algebra. We have
$$
0=x^3-1=(x-1)(x^2+x+1).
$$
Since $x \notin \mathbf{F}$ and $A$ is a division algebra, we get that $x^2+x+1=0.$ So we get a quadratic extension $\mathbf{F} \subset \mathbf{K}$ in $A$ generated by $x.$ This means that
$$
9=\dim_{\mathbf{F}}A=\dim_{\mathbf{F}}\mathbf{K} \cdot \dim_{\mathbf{K}}A=2\dim_{\mathbf{K}}A,
$$ 
which is a contradiction.

\end{proof}

\begin{Lemma}\label{lemma:csaomega}
Let $A$ be a central simple algebra of dimension $9$ over a field $\mathbf{F}$ of characteristic different from $3.$ Assume that $\widehat{T}=(\mathbb{Z}/3\mathbb{Z})^2 \subset A^*/\mathbf{F}^*,$ and let   $x$ and~$y$ be  generators of $\widehat{T}.$ If $\bar{x}$ and $\bar{y}$  are any elements in the preimages of $x$ and $y$  under the natural homomorphism~$A^* \to A^*/\mathbf{F}^*,$ respectively, then 
$$
\bar{x}\bar{y}\bar{x}^{-1}=\omega\bar{y},
$$
\noindent where $\omega$ is a non-trivial cube root of unity.
\end{Lemma}

\begin{proof}
We have $xyx^{-1}y^{-1}=1$ in $A^*/\mathbf{F}^*,$ because $x$ and $y$ commute with each other. Therefore, we obtain  
\begin{equation}\label{eq:csaomegaa}
\bar{x}\bar{y}\bar{x}^{-1}=a\bar{y}
\end{equation}
\noindent for some $a \in \mathbf{F}^*.$ As $\bar{x}^3 \in \mathbf{F}^*,$ we get that $\bar{x}^3$ commutes with $\bar{y},$ so by \eqref{eq:csaomegaa}
$$
\bar{y}=\bar{x}^3\bar{y}\bar{x}^{-3}=a^3\bar{y}.
$$
\noindent Thus, we have $a^3=1.$

By contradiction, assume that $a=1.$  So the elements $\bar{x}$ and $\bar{y}$ commute with each other. Assume that $\bar{x}^3$ and $\bar{y}^3$ belong to $(\mathbf{F}^*)^3.$ So without loss of generality we can take $\bar{x}$ and $\bar{y}$ in the preimages of $x$ and $y,$ respectively, such that $\bar{x}^3=\bar{y}^3=1$ and $\bar{x}$ and $\bar{y}$ also commute with each other. This means that $(\mathbb{Z}/3\mathbb{Z})^2 \subset A^*.$ By Lemma~\ref{lemma:matrixalgebraifscalar} we get that $A$ is a matrix algebra. So by Lemma~\ref{lemma:representationz/3z}\ref{lemma:representationz/3z23} we get that  
$$
(\mathbb{Z}/3\mathbb{Z})^2 \not\subset A^*.
$$
Therefore, in this case $a \neq 1.$

Now assume that  $\bar{x}^3 \notin (\mathbf{F}^*)^3.$ Consider the subalgebra $B$ of $A$ generated by the elements
\begin{equation}\label{eq:elementsofB}
1,\bar{x}, \bar{y}, \bar{x}^2, \bar{y}^2, \bar{x}\bar{y}, \bar{x}^2\bar{y}, \bar{x}\bar{y}^2, \bar{x}^2\bar{y}^2.
\end{equation}

 Let us prove that these elements are linearly independent over $\mathbf{F}.$ First of all, note that the elements $1$, $\bar{x}$ and $\bar{x}^2$ are linearly independent over $\mathbf{F}.$ Indeed, let us consider the algebra $B'$ generated by these three elements. Then from the equation
$$
\dim_{B'}A \cdot \dim_{\mathbf{F}}B'=\dim_{\mathbf{F}}A
$$
\noindent we get that $\dim_{\mathbf{F}}B'$ is either $1,$ or $3.$ However, by assumptions the element $\bar{x}$ does not lie in $\mathbf{F}.$ As
$$
A \supseteq B \supseteq B',
$$
\noindent we get that $\dim_{\mathbf{F}}B$ is either $3,$ or $9.$ So it is enough to prove that $B' \neq B$ to exclude the case $\dim_{\mathbf{F}}B=3.$ To prove this it is enough to prove that $\bar{y}$ does not lie in $B'.$ Assume the contrary. Then by Lemma~\ref{lemma:fielddegree3} the element $\bar{y}$ is either $c$, or~$c\bar{x},$ or $c\bar{x}^2.$ That is the contradiction with the fact that $x$ and $y$ generate the group $\left(\mathbb{Z}/3\mathbb{Z}\right)^2.$

 This gives us  $\dim_{\mathbf{F}}B=9=\dim_{\mathbf{F}}A.$ But~$A$ is a central simple algebra and $B$ is a commutative algebra. This contradiction gives us that $a$ is a non-trivial cube root of unity.

\end{proof}

\begin{Lemma}\label{lemmaSB}
Let  $\mathbf{F}$ be a field of characteristic different from $3$  which does not contain non-trivial cube roots of unity. Let $V$ be a  Severi--Brauer surface over $\mathbf{F}.$  Then $(\mathbb{Z}/3\mathbb{Z})^2$ is not a subgroup of $\mathrm{Aut}(V)$.
\end{Lemma}

\begin{proof}
 Let $A$ be a central simple algebra corresponding to the Severi--Brauer surface~$V.$ Then by Theorem \ref{csaprop} we obtain~\mbox{$\mathrm{Aut}(V) \simeq A^*/\mathbf{F}^*.$}  Assume that
$$
(\mathbb{Z}/3\mathbb{Z})^2 \subset A^*/\mathbf{F}^*.
$$
\noindent  Let~$x$ and $y$ be generators of $(\mathbb{Z}/3\mathbb{Z})^2.$ Let~$\bar{x}$ and~$\bar{y}$ be any elements in the preimages of $x$ and $y$  under the natural homomorphism~$A^* \to A^*/\mathbf{F}^*,$ respectively.  By Lemma~\ref{lemma:csaomega} we obtain
$$
\bar{x}\bar{y}\bar{x}^{-1}=\omega\bar{y}
$$
\noindent where $\omega$ is a non-trivial cube root of unity. However, the field $\mathbf{F}$ does not contain non-trivial cube root of unity. This contradiction gives us 
$$
(\mathbb{Z}/3\mathbb{Z})^2 \not\subset \mathrm{Aut}(V).
$$

\end{proof}

\begin{Cor}[cf.~{\cite[Corollary 6.3]{Shramov1}}]\label{cor:T2}
Let $V$ be a non-trivial Severi--Brauer surface over a field $\mathbf{F}$ of characteristic different from $3.$ Let $\widehat{T}=(\mathbb{Z}/3\mathbb{Z})^2$  be a subgroup in the automorphism group of~$V$ generated by the elements $x$ and $y.$ Then the sets of fixed points of $x$ and $y$ are disjoint.
\end{Cor}

\begin{proof}
According to Lemma~\ref{lemmaSB} the field $\mathbf{F}$ contains a non-trivial cube root of unity. Let $A$ be a central simple algebra corresponding to $V.$ Then by  Theorem \ref{csaprop}  we have~$\widehat{T} \subset A^*/\mathbf{F}^*.$ Therefore, by Lemma~\ref{lemma:csaomega} for any   elements $\bar{x}$ and~$\bar{y}$ in the preimages of $x$ and $y$  under the natural homomorphism~$A^* \to A^*/\mathbf{F}^*,$ respectively we get
\begin{equation}\label{eq:corT2}
\bar{x}\bar{y}=\omega\bar{y}\bar{x},
\end{equation}
\noindent where $\omega$ is a non-trivial cube root of unity. Let us consider the elements $x$ and~$y$ as automorphisms of $V_{\overline{\mathbf{F}}} \simeq \mathbb{P}^2_{\overline{\mathbf{F}}}.$ This means that $x$ and $y$ are considered as elements in~$\mathrm{PGL}_3(\overline{\mathbf{F}}),$ and $\bar{x}$ and $\bar{y}$ are considered as elements in $\mathrm{GL}_3(\overline{\mathbf{F}})$ such that they  satisfy relation~\eqref{eq:corT2}. Thus, $\bar{x}$ and $\bar{y}$ cannot have a common eigenvector and hence, the elements $x$ and $y$  have no common fixed points. 

\end{proof}

It turns out that the action of the group $\mathbb{Z}/3\mathbb{Z}$ on a non-trivial Severi--Brauer surface $V,$ which exists by Lemma~\ref{lemma:autz/3z}, can be lifted to a smooth cubic surface obtained as a blowup of $V.$

\begin{Lemma}\label{lemma:p16}
Let $\mathbf{F}$ be a field of characteristic different from $3,$ and let $V$ be  a non-trivial Severi--Brauer surface over $\mathbf{F}.$ Let $T \simeq \mathbb{Z}/3\mathbb{Z}$ be a subgroup of $\mathrm{Aut}(V).$ Then

\begin{enumerate}
\renewcommand\labelenumi{\rm (\roman{enumi})}
\renewcommand\theenumi{\rm (\roman{enumi})}

\item\label{p123}
there is a unique triple of geometric points $p_1$, $p_2$, $p_3$ on $V$ which are fixed by the group~$T$ (in particular, the triple $\{p_1, p_2,p_3\}$ is defined over $\mathbf{F}$). Moreover, over $\overline{\mathbf{F}}$ these $3$ points do not lie on a line;

\item\label{p456}
there is an orbit of $T$ consisting of a triple of geometric points $p_4$, $p_5$, $p_6,$ which is defined over the field $\mathbf{F}.$  Moreover, over $\overline{\mathbf{F}}$ these $3$ points do not lie on a line;

\item\label{p16cubic}
for any choice of $p_4$, $p_5$, $p_6$ as in \ref{p456} \emph{the blowup of $V$ at $p_1, \ldots, p_6$ is a smooth cubic surface.} 
 \end{enumerate}

\end{Lemma}

\begin{proof}
First of all, let us prove \ref{p123}. Consider the action of $T$ on $V_{\overline{\mathbf{F}}}.$ Since $T$ is a cyclic group of finite order, it is diagonalizable in $\mathrm{PGL}_3(\overline{\mathbf{F}}).$ Then the set of fixed points of $T$ is either $3$ points $p_1,$ $p_2$ and $p_3$, or an isolated point $p$ and a line $l.$ But the last case is impossible, because this means that $p$ is defined over~$\mathbf{F},$ which is impossible. So let us consider $3$ geometric fixed points $p_1,$ $p_2$ and $p_3$ on~$V.$ Observe that such three geometric points do not lie on one line $l$ in $V_{\overline{\mathbf{F}}} \simeq \mathbb{P}^2_{\overline{\mathbf{F}}}$. Indeed, otherwise we get that the action of $T$ on $l$ has exactly $3$ fixed points. But this means that~$T$ fixes~$l$ pointwise over $\overline{\mathbf{F}}.$ Since $T$ commutes with $\mathrm{Gal}(\overline{\mathbf{F}}/\mathbf{F}),$ this means that $l$ is defined over  $\mathbf{F}$, which is impossible by~\cite[Theorem 5.1.3]{SimpleAlgebras}.  The triple~$\{p_1, p_2, p_3\}$ is defined over $\mathbf{F}$ because the Galois group~$\mathrm{Gal}(\overline{\mathbf{F}}/\mathbf{F})$ commutes with $T.$ 

Now let us prove \ref{p456}. By~\cite[Theorem 2]{Trepalin1} (note that the result of the mentioned theorem just need characteristic different from $3$), the quotient~$V/T$ is $\mathbf{F}$-rational. As the set of $\mathbf{F}$-points in $V/T$ is a dense subset, the preimage under the quotient morphism 
$$
V \to V/T
$$
\noindent  of a general $\mathbf{F}$-point $p \in V/T$ consists of $3$ geometric points forming one $\mathrm{Gal}(\overline{\mathbf{F}}/\mathbf{F})$-orbit. Denote these geometric points by $p_4$, $p_5$ and $p_6.$ Note that they do not lie on one line $l$ in $V_{\overline{\mathbf{F}}},$ because this means that $l$ is defined over $\mathbf{F}.$ By~\cite[Theorem~5.1.3]{SimpleAlgebras} there is no $\mathbf{F}$-lines on the non-trivial Severi--Brauer surface $V.$

Let us prove \ref{p16cubic}. We have to show that the blow up of the points~$p_1, \ldots, p_6$ is a smooth cubic surface. For this it is enough to prove that any $3$ points among  these points do not lie on a line and not all $6$ points  lie on a conic. First of all, let us prove that any triple of the set $p_1, \ldots, p_6$ does not lie on one line. Indeed, first of all, assume that $p_1, p_2$ and $p_4$ lie on the line $l$ over $\overline{\mathbf{F}}.$  Then as~$T$ permutes $p_4$,~$p_5$ and~$p_6$ and fixes $p_1$ and $p_2,$ these points also lie on $l.$ But this is impossible by the above argument.

Now assume that $p_4$, $p_5$ and $p_1$ lie on one line over $\overline{\mathbf{F}}.$ Then under the action of a non-trivial element $\alpha$ of the group $T$ this line maps  to the line passing through~$p_1,$~$p_6$ and one of the points $p_4$ and $p_5,$ which means that $p_4,$ $p_5$ and $p_6$ lie on one line, which contradicts the above arguments. 

Finally, note that no non-trivial automorphism fixes $3$ points on a conic. Therefore, the points~\mbox{$p_1$, $ p_2$, $ p_3$, $ p_4$, $ p_5$, $ p_6$} do not lie on one conic. So the blowup of these~$6$ points gives us a smooth cubic surface.

\end{proof}

\begin{Remark}\label{remark:p16}
Let $\mathbf{F}$ be a field of characteristic different from $3.$  Let $V \simeq \mathbb{P}^2_{\mathbf{F}}.$  There is a biregular action of~\mbox{$T \simeq \mathbb{Z}/3\mathbb{Z}$} on $V$  which is generated by an element cyclically permuting the coordinates. The geometric points 
\begin{equation}\label{eq:projsurf123}
p_1=[1:1:1], \quad p_2=[\omega:1:\omega^2], \quad p_3=[\omega^2:1:\omega],
\end{equation}
\noindent where $\omega$ is a non-trivial cube of unity, are fixed by $T.$ Note that the union $p_1 \cup p_2 \cup p_3$ is $\mathrm{Gal}(\overline{\mathbf{F}}/\mathbf{F})$-invariant.  The geometric points
\begin{equation}\label{eq:projsurf456}
p_4=[1:0:0], \quad p_5=[0:1:0], \quad p_6=[0:0:1]
\end{equation}
\noindent form an orbit of $T.$ It is not hard to see that any $3$ points among these $6$ ones do not lie on a line and not all $6$ points  lie on a conic. Therefore, the blowup of~\mbox{$p_1$, $ p_2$, $ p_3$, $ p_4$, $ p_5$, $ p_6$} is a smooth cubic surface.
\end{Remark}

If the field $\mathbf{F}$ contains a non-trivial cube root of unity then by Lemma \ref{lemma:autz/3z} on any  Severi--Brauer surface over $\mathbf{F}$ there is a biregular action of $(\mathbb{Z}/3\mathbb{Z})^2.$ It turns out that this action can be lifted to a smooth cubic surface which is a blowup of $V$ provided that $V$ is a non-trivial Severi--Brauer surface.

\begin{Lemma}\label{lemma:p16omega}
Let $\mathbf{F}$ be a field of characteristic different from $3$ and let $V$ be a non-trivial Severi--Brauer surface over $\mathbf{F}.$ Assume that the group~\mbox{$\widehat{T} \simeq (\mathbb{Z}/3\mathbb{Z})^2$}  is contained in~$\mathrm{Aut}(V).$ Let $b$ and $c$ be generators of $\widehat{T}.$ Then

\begin{enumerate}
\renewcommand\labelenumi{\rm (\roman{enumi})}
\renewcommand\theenumi{\rm (\roman{enumi})}

\item\label{fix123}
there is a unique triple of geometric points $p_1$, $p_2$, $p_3$ on $V$ which are fixed by the subgroup generated by $b$ (in particular, the triple $\{p_1,p_2,p_3\}$ is defined over $\mathbf{F}$);

\item\label{fix456}
there is a unique triple of geometric points $p_4$, $p_5$, $p_6$ on $V$ which are fixed by the subgroup generated by $c$ (in particular, the triple $\{p_4,p_5,p_6\}$ is defined over $\mathbf{F}$);

\item\label{intersection16}
the points $p_1$, $p_2$, $p_3$, $p_4$, $p_5$, $p_6$ are distinct;

\item\label{triplesomegab}
the element $b$  cyclically permutes $p_4$, $p_5$, $p_6$; 

\item\label{tripleomegac}
the element $c$  cyclically permutes $p_1$, $p_2$, $p_3$;

\item\label{p16cubicomega}
the blowup of $V$ at $p_1, \ldots, p_6$ is a smooth cubic surface.
 \end{enumerate}

\end{Lemma}

\begin{proof}
Assertions \ref{fix123} and \ref{fix456} follow directly from Lemma \ref{lemma:p16}\ref{p123}. Assertion \ref{intersection16} follows from Corollary \ref{cor:T2}. Assertions \ref{triplesomegab} and \ref{tripleomegac} follow from the fact that  $b$ and~$c$ commute with each other which means that the element $b$ fixes the set of fixed points of $c$ and vice versa. Assertion \ref{p16cubicomega} follows from Lemma \ref{lemma:p16}\ref{p16cubic}.

\end{proof}

\begin{Remark}\label{remark:p16omega}
Assume that $V \simeq \mathbb{P}^2$ over a field $\mathbf{F}$ of characteristic different from~$3$ which contains a non-trivial cube root $\omega$ of unity. Then there is a biregular action of $\mathbb{Z}/3\mathbb{Z}$ on $V$ which was constructed in Remark \ref{remark:p16}. Let $b$ be a generator of this group. Consider the element 
$$
c=
\begin{pmatrix}
\omega & 0 & 0 \\
0  & 1 & 0\\
0 & 0 & \omega^2
\end{pmatrix}
\in \mathrm{PGL}_3(\mathbf{F}).
$$ 
\noindent  Together with $b$ it generates the subgroup $\widehat{T}=(\mathbb{Z}/3\mathbb{Z})^2$ in $\mathrm{Aut}(V).$ While the element $b$   fixes~$p_1$,~$p_2$ and $p_3$ from \eqref{eq:projsurf123} and cyclically permutes $p_4$,~$p_5$ and $p_6$ from~\eqref{eq:projsurf456},  the element $c$ on the contrary  fixes~$p_4$,~$p_5$ and $p_6$  and cyclically permutes~$p_1$,~$p_2$ and~$p_3$. In particular, the set $\{p_1, \ldots, p_6\}$ is $\widehat{T}$-invariant. It is straightforward to check that the blowup of $V$ at these $6$ points is a smooth cubic surface.

\end{Remark}

\section{Cubic surfaces}\label{sectionofcubic}
In this section we study cubic surfaces over a field of characteristic zero. First of all, we make the following observation.

\begin{Lemma}\label{lemma:cubic^3}
Let $\mathbf{F}$ be a field of  characteristic different from $3$ which does not contain a non-trivial cube root of unity. Let $S$ be a smooth cubic surface over $\mathbf{F}.$ Then the group~$(\mathbb{Z}/3\mathbb{Z})^3$ does not act biregularly  on $S.$
\end{Lemma}

\begin{proof}
Assume that there is an action of $(\mathbb{Z}/3\mathbb{Z})^3$ on $S.$ Then we get an induced action of this group on $\mathbb{P}^3.$ Also the action of $(\mathbb{Z}/3\mathbb{Z})^3$ induces a linear action on 
$$
H^0(S,-K_S) \simeq \mathbf{F}^4.
$$
\noindent  However, by Lemma \ref{lemma:representationz/3z}\ref{lemma:representationz/3z34} the group $(\mathbb{Z}/3\mathbb{Z})^3$ does not act linearly on $\mathbf{F}^4.$

\end{proof}

The following lemma states that any map between smooth cubic surfaces is defined uniquely by the image of pairwise skew lines  $E_1, \ldots, E_6.$

\begin{Lemma}\label{lemma:uniquenessaut}
Let $\mathbf{F}$ be an algebraically closed field. Let $S \subset \mathbb{P}^3$ be a smooth cubic surface with pairwise skew lines $E_1, \ldots, E_6.$  Assume that there is an element~\mbox{$\theta \in \mathrm{PGL}_4(\mathbf{F})$} such that $\theta(E_i)=E_i$ for all $i.$ Then $\theta=id.$
\end{Lemma}

\begin{proof}
First of all, assume that $\theta$ preserves $S.$ Then we can blow down $E_1, \ldots, E_6$ and get the induced automorphism $\theta$ on $\mathbb{P}^2$ with $6$ fixed points in general position. Therefore, $\theta$ acts trivially on $\mathbb{P}^2$ and so on $S$ and $\mathbb{P}^3.$

Now assume that $\theta(S) \neq S.$ Denote by $S'$ the image of cubic surface $\theta(S).$ The intersection $S \cdot S'$ of these two cubic surfaces  in $\mathbb{P}^3$ is a possibly non-reduced curve of degree $9.$ Our~$6$ lines $E_1, \ldots, E_6$ are contained in this curve. It is well-known that for any $5$ lines among $E_1, \ldots, E_6$  there is a unique line on $S$ which intersects all these $5$ lines. This line is a strict transform of the conic through $5$ points in general position which is unique. Any of these lines lies in $S'$ because it has at least~$5$ common points with $S'.$ And so we get that the curve of degree $9$ contains~$12$ lines, which is a contradiction.

\end{proof}

Let $S$ be a smooth  cubic surface  over a field $\mathbf{F}.$ Then there is a natural action of the Weyl group $W(\mathrm{E}_6)$ on the Picard group  $\mathrm{Pic}(S_{\overline{\mathbf{F}}})$ of $S_{\overline{\mathbf{F}}}$  (see, for example,~\mbox{\cite[Corollary~25.1.1]{Manin}).} Namely, the group $W(\mathrm{E}_6)$ consists of all  automorphisms of the lattice~\mbox{$\mathrm{Pic}(S_{\overline{\mathbf{F}}}) \simeq \mathbb{Z}^7$} preserving the intersection form and fixing the canonical class~$K_S.$ In particular, for every choice of $6$ pairwise skew lines $E_1, \ldots, E_6$ on $S_{\overline{\mathbf{F}}}$ there is an action of the symmetric group $\mathrm{S}_6$ on the Picard group $\mathrm{Pic}(S_{\overline{\mathbf{F}}})$ which permutes the classes of these $6$ lines. It is well-known that the automorphism group of $S$ is embedded in the Weyl group~$W(\mathrm{E}_6)$ (see~\cite[Corollary~8.2.40]{DolgachevAlg}). Note also that the Galois group $\mathrm{Gal}(\overline{\mathbf{F}}/\mathbf{F})$ maps to $W(\mathrm{E}_6).$

\begin{Lemma}\label{lemma:alphagal}
Let $S$ be a smooth cubic surface over a field $\mathbf{F}.$  Let $\phi \in \mathrm{Aut}(S_{\overline{\mathbf{F}}})$ be an automorphism of $S_{\overline{\mathbf{F}}}$ which  commutes with the image of the Galois group $\mathrm{Gal}(\overline{\mathbf{F}}/\mathbf{F})$ in~$W(\mathrm{E}_6).$ Then the automorphism $\phi$ is defined over $\mathbf{F},$ i.e.~\mbox{$\phi \in \mathrm{Aut}(S).$}
\end{Lemma}

\begin{proof}
Denote by $E_i$ for $1 \leqslant i \leqslant 27$ all lines lying on $S_{\overline{\mathbf{F}}}.$ Then the curve 
$$
\mathcal{E}=E_1+\ldots+E_{27}
$$
\noindent is Galois-invariant, thus, is defined over $\mathbf{F}.$  Applying Lemma~\ref{lemma:irreducibleGalois} to the curve $\mathcal{E}$ and the element~\mbox{$\phi \in \mathrm{PGL}_4(\overline{\mathbf{F}})$} we obtain $\psi \in \mathrm{PGL}_4(\mathbf{F})$ such that $\phi(E_i)=\psi(E_i)$ for all~\mbox{$1 \leqslant i \leqslant 27.$}  Hence, applying  Lemma \ref{lemma:uniquenessaut} to the element $\phi \circ \psi^{-1}$ we get $\phi=\psi.$ Thus, we have $\phi \in \mathrm{Aut}(S).$

\end{proof}

We are mostly interested in two conjugacy classes of elements in $W(\mathrm{E}_6),$ which are conjugacy classes of type $\mathrm{A}_2$ and $\mathrm{A}_2^2$ in the notation of~\cite{Carter}. They consist of the elements whose eigenvalues on the vector space corresponding  to the root system~$\mathrm{E}_6$  are
 $$
\omega, \omega^2,1,1,1,1
$$
\noindent and
 $$
\omega, \omega, \omega^2, \omega^2,1,1,
$$
\noindent respectively; here $\omega$ is a non-trivial cube root of unity.  We will say that  an automorphism of a smooth cubic surface is of type $\mathrm{A}_2$  (of type $\mathrm{A}_2^2$), if its image in  the Picard group is an element in the conjugacy class of type $\mathrm{A}_2$ (of type $\mathrm{A}_2^2$).

\begin{Example}\label{exampleA2}
Let $S$ be a smooth cubic surface. Let $E_1, \ldots, E_6$ be pairwise skew geometric lines on $S.$ Let us consider the element~$(456)$ in $W(\mathrm{E}_6)$ which cyclically permutes $E_4$, $E_5$ and $E_6$ and fixes  $E_1$, $E_2$ and $E_3.$ Then by~\mbox{\cite[Table 1]{RybakovTrepalin}}  this element is of type $\mathrm{A}_2.$
\end{Example}

Now let us discuss the automorphism group of the Fermat cubic surface, i.e. the cubic surface which is defined by the equation $x^3+y^3+z^3+t^3=0.$

\begin{Example}\label{exampleA2A22}
Let $S$ be the Fermat cubic surface over an algebraically closed field~$\mathbf{F}$ of characteristic different from $3.$  Then by~\cite[Lemma 2.4 and Table 1]{RybakovTrepalin}  an element in~$\mathrm{Aut}(S)$ is of type~$\mathrm{A}_2$ if and only if its quadruple of eigenvalues in~$\mathrm{PGL}_4(\mathbf{F}),$ up to multiplication by a non-zero element in the field $\mathbf{F},$ is of the form $1,1,\omega,\omega,$ where~$\omega$ is a non-trivial cube root of unity. An element in $\mathrm{Aut}(S)$ is of type $\mathrm{A}_2^2$ if and only if its quadruple of eigenvalues in $\mathrm{PGL}_4(\mathbf{F}),$ up to multiplication by a non-zero element in the field~$\mathbf{F},$ is of the form $1,1,\omega,\omega^2.$

\end{Example}

\begin{Th}[{\cite[Lemma 10.14]{DolgachevDuncan} and~\cite[Theorem~6.10]{DI}}]\label{thofDolg}
Let $S$ be a smooth cubic surface over an algebraically closed field of characteristic different from $3$ admitting an automorphism of type $\mathrm{A}_2.$ Then $S$ is isomorphic to the Fermat cubic. 

\end{Th}

\begin{Cor}\label{lemma:Farmatcubic}
Let $S$ be a smooth cubic surface over an algebraically closed field of characteristic different from $3.$ Let $E_1,\ldots, E_6$ be pairwise skew lines on $S.$ Assume that there is a biregular action of the group $T\simeq \mathbb{Z}/3\mathbb{Z}$ on $S$ which fixes $E_1$, $E_2$ and~$E_3$ and cyclically permutes $E_4$, $E_5$ and $E_6$. Then $S$ is isomorphic to the Fermat cubic.

\end{Cor}

\begin{proof}
Consider the subgroup $\mathrm{S}_6 \subset W(\mathrm{E}_6)$ acting on $E_1, \ldots, E_6$ by permutations. Then the image of the group $T$ in $W(\mathrm{E}_6)$ is generated by the element $(456) \in \mathrm{S}_6.$  By Example~\ref{exampleA2} the conjugacy class of the element $(456)$ has type $\mathrm{A}_2.$  Therefore, by Theorem~\ref{thofDolg} the cubic surface $S$ is isomorphic to the Fermat cubic.

\end{proof}

\begin{Lemma}\label{lemma:a2inaut}
Let $S$ be the Fermat cubic surface over an algebraically closed field of characteristic different from $2$ and $3.$ Then there are exactly $6$ elements in $\mathrm{Aut}(S)$ of type~$\mathrm{A}_2$ and they commute with each other. Moreover, the centralizer of any element of type $\mathrm{A}_2$ in~$\mathrm{Aut}(S)$ is isomorphic to $(\mathbb{Z}/3\mathbb{Z})^3 \rtimes (\mathbb{Z}/2\mathbb{Z})^2.$
\end{Lemma}

\begin{proof}
The Fermat cubic surface $S$ is defined by the equation \mbox{$x^3+y^3+z^3+t^3=0$} in~$\mathbb{P}^3.$ The automorphism group of $S$ is isomorphic to $(\mathbb{Z}/3\mathbb{Z})^3 \rtimes \mathrm{S}_4$ (see, for instance,~\cite[Theorem 9.5.6]{DolgachevAlg} and~\cite[Lemma 5.1]{DolgachevDuncan}). The group $\mathrm{S}_4$ acts by the permutations of the coordinates $x,$ $y,$ $z,$ and $t.$ The group $(\mathbb{Z}/3\mathbb{Z})^3$ acts by the multiplication of the coordinates by cube roots of unity. Any element in $(\mathbb{Z}/3\mathbb{Z})^3$ can be written as
\begin{equation}\label{eq:repFermataut}
\begin{pmatrix}
1 & 0 & 0 & 0\\
0 & \omega^a & 0 & 0\\
0 & 0 & \omega^b & 0\\
0 & 0 & 0 & \omega^c
\end{pmatrix} \in \mathrm{PGL}_{4}(\overline{\mathbf{F}}),
\end{equation}
\noindent where $\omega$ is a non-trivial cube root of unity and $a,b,c \in \{0,1,2\}.$ From Example~\ref{exampleA2A22} we get that the element of the form  \eqref{eq:repFermataut} is of type $\mathrm{A}_2$  if and only if 
$$
a=b\neq 0, \; c=0; \quad \text{or} \quad a=c\neq 0, \; b=0; \quad \text{or} \quad b=c\neq 0, \; a=0.
$$

\noindent Therefore, these are $6$ elements of type $\mathrm{A}_2$ in $\mathrm{Aut}(S),$ and we will show now that there are not more.

By Example~\ref{exampleA2A22}  all elements of order $3$ in the group $\mathrm{S}_4$ are of type~$\mathrm{A}_2^2,$ because the eigenvalues of the corresponding matrices are $\omega, \omega^2,1,1.$ Let us consider all elements in $\mathrm{Aut}(S)$ of order $3$  of the form $gh,$ where~\mbox{$g \in (\mathbb{Z}/3\mathbb{Z})^3$} and $h \in \mathrm{S}_4.$ Assume that $h$ is non-trivial. So $h$ is of order $3.$ Since all elements of order $3$ in~$\mathrm{S}_4$ are conjugate, without loss of generality we can put $h=(234).$ Let us find all $g \in (\mathbb{Z}/3\mathbb{Z})^3$ such that $gh$ is of order $3.$ Since the matrix of $h$ is of the form
$$
\begin{pmatrix}
1 & 0 & 0 & 0\\
0 & 0 & 0 & 1\\
0 & 1 & 0 & 0\\
0 & 0 & 1 & 0
\end{pmatrix} \in \mathrm{PGL}_{4}(\overline{\mathbf{F}}),
$$
\noindent by direct computation we get that $g$ is of the form \eqref{eq:repFermataut} such that
\begin{equation}\label{eq:a+b+c=0mod3}
a+b+c \equiv 0 \mod 3.
\end{equation}
\noindent So again by direct computation one can find that the characteristic polynomial of~$gh$ such that $g$ satisfies \eqref{eq:a+b+c=0mod3} is $(\lambda-1)(\lambda^3-1).$ So the eigenvalues of such $gh$ are 
$$
1,1,\omega,\omega^2.
$$
\noindent  Therefore, by Example~\ref{exampleA2A22} such elements are of type $\mathrm{A}_2^2.$ So there are exactly $6$ elements of type~$\mathrm{A}_2$ in $\mathrm{Aut}(S),$  and they commute with each other.

Let us prove that the centralizer of any element of type $\mathrm{A}_2$ in $\mathrm{Aut}(S)$ is isomorphic to~$(\mathbb{Z}/3\mathbb{Z})^3 \rtimes (\mathbb{Z}/2\mathbb{Z})^2.$ Indeed, as we showed above all such elements lie in the normal subgroup $(\mathbb{Z}/3\mathbb{Z})^3 \subset \mathrm{Aut}(S).$ Let us fix the element $s \in \mathrm{Aut}(S)$ of type~$\mathrm{A}_2$ and consider the elements in $\mathrm{S}_4$ which commute with $s.$ These are the elements  which permute the eigenvectors of $s$ with the same eigenvalues. Recall from Example \ref{exampleA2A22} that  eigenvalues in $\mathrm{PGL}_4(\mathbf{F})$ up to multiplication by non-zero element of $s$ are $1,1,\omega,\omega.$ Thus, the centraliser is isomorphic to $(\mathbb{Z}/3\mathbb{Z})^3 \rtimes (\mathbb{Z}/2\mathbb{Z})^2.$

\end{proof}

\begin{Lemma}\label{lemma:Sylow}
Let $S$ be a smooth cubic surface over a field $\mathbf{F}$ of characteristic different from $2$ and $3.$  Let $E_1, \ldots, E_6$ be pairwise skew lines on $S_{\overline{\mathbf{F}}}$ such that their union is Galois-invariant. Suppose that there is an element~$b \in \mathrm{Aut}(S)$ of order~$3$ such that it fixes~$E_1$,~$E_2$ and $E_3$ and cyclically permutes $E_4$, $E_5$ and~$E_6.$ Then there is a unique $3$-Sylow subgroup of the centralizer of the element $b$ in $\mathrm{Aut}(S_{\overline{\mathbf{F}}}),$ respectively $W(\mathrm{E}_6),$ and in both cases, it is isomorphic to $(\mathbb{Z}/3\mathbb{Z})^3.$ Moreover, both these groups coincide.

\end{Lemma}

\begin{proof}
By Corollary \ref{lemma:Farmatcubic} we get that $S_{\overline{\mathbf{F}}}$ is the Fermat cubic surface. Therefore, we have $\mathrm{Aut}(S_{\overline{\mathbf{F}}}) \simeq (\mathbb{Z}/3\mathbb{Z})^3 \rtimes \mathrm{S}_4.$ By Lemma \ref{lemma:a2inaut} the centralizer $C$ of the element~$b$ in~$\mathrm{Aut}(S_{\overline{\mathbf{F}}})$ is isomorphic to  $(\mathbb{Z}/3\mathbb{Z})^3 \rtimes (\mathbb{Z}/2\mathbb{Z})^2,$ since by Example~\ref{exampleA2} the element~$b$ is of type~$\mathrm{A}_2.$ By Sylow' theorems there is a unique $3$-Sylow subgroup $Z$ in $C,$ namely 
$$
Z \simeq (\mathbb{Z}/3\mathbb{Z})^3.
$$
\noindent   In particular, it is normal in $C$ and is preserved by the automorphism group of $C.$

 Let us find the order of the centralizer $\mathcal{G}$ for the element $b$ in $W(\mathrm{E}_6).$ By Example~\ref{exampleA2} the element $b$ is of type $\mathrm{A}_2.$ By the classification of conjugacy classes of elements in $W(\mathrm{E}_6)$ (see~\cite[Table 9]{Carter}) the number of elements in the conjugacy class of the element $b$ is equal to $240.$ Therefore, the order of the centralizer is equal to 
$$
\frac{51840}{240}=216=2^3 \cdot 3^3.
$$

Let us prove that $\mathcal{G}$ contains a unique Sylow $3$-subgroup. Indeed, we have $\mathcal{G} \supset C,$ because $\mathrm{Aut}(S_{\overline{\mathbf{F}}}) \subset W(\mathrm{E}_6).$ The subgroup $C$ is normal in $\mathcal{G}$ because its index is~$2.$ In particular, $C$ is preserved by conjugation of elements in $\mathcal{G}.$ Thus, the unique $3$-Sylow subgroup $Z$ of $C$ from above is preserved by conjugation in $\mathcal{G}.$ So by~\cite[Exercise~8a), Section 4.4]{Foote} the subgroup $Z$ is normal in $\mathcal{G}$ as well. As $Z$ is a Sylow subgroup in $\mathcal{G}$ we get that it is a unique Sylow $3$-subgroup. 
Hence, $Z \simeq (\mathbb{Z}/3\mathbb{Z})^3$ is a unique $3$-Sylow subgroup of both $\mathcal{G}$ and $C.$

\end{proof}

Let us prove the auxiliary proposition about endomorphisms of the Picard group of a smooth cubic surface  which is needed for the lemmas below.

\begin{Prop}\label{prop:WE6}
Let $S$ be a smooth cubic surface over an algebraically closed field~$\mathbf{F}$  and set $\mathrm{Pic}(S)_{\mathbb{Q}}=\mathrm{Pic}(S)\otimes {\mathbb{Q}}.$ Let $\alpha$ be an endomorphism of  the $\mathbb{Q}$-vector space $\mathrm{Pic}(S)_{\mathbb{Q}}$  such that $\alpha$ fixes the canonical class $K_S$ and  maps pairwise skew lines $E_1, \ldots, E_6$ to some pairwise skew lines $\alpha(E_1), \ldots, \alpha(E_6).$ Then~$\alpha$ is an automorphism of~$\mathrm{Pic}(S)_{\mathbb{Q}}$ which restricts to an automorphism of~$\mathrm{Pic}(S)$ and preserves the intersection form on $\mathrm{Pic}(S).$ In particular, it  lies in $W(\mathrm{E}_6).$
\end{Prop}

\begin{proof}
 The divisors~$K_S, E_1, \ldots, E_6$ and $K_S, \alpha(E_1), \ldots, \alpha(E_6)$ are two bases of the vector space~$\mathrm{Pic}(S)_\mathbb{Q},$ so $\alpha$  is an automorphism of $\mathrm{Pic}(S)_\mathbb{Q}.$ Since any divisor  $\alpha(E_i)$ for all $i=1, \ldots, 6$ is $(-1)$-curves, the automorphism $\alpha$  preserves the intersection form on $\mathrm{Pic}(S)_{\mathbb{Q}}.$  The Picard group~$\mathrm{Pic}(S)$ is a lattice generated by~$E_1, \ldots, E_6$ and the divisor 
$$
L=\frac{1}{3}\left(-K_S+E_1+\ldots+E_6 \right),
$$
\noindent which is a pull-back of a line via the morphism $\pi:S \to \mathbb{P}^2$ contracting $E_1, \ldots, E_6.$ The automorphism $\alpha$  maps $L$ to
$$
\alpha(L)=\frac{1}{3}\left(-K_S+\alpha(E_1)+\ldots+\alpha(E_6 \right)).
$$

\noindent We can blow down the skew lines $\alpha(E_1), \ldots, \alpha(E_6)$ and obtain a map~\mbox{$\widetilde{\pi}:S \to \mathbb{P}^2$}  such that $\alpha(L)=\widetilde{\pi}^*l,$ where $l$ is a line on $\mathbb{P}^2.$ Therefore, the Picard group of~$S$ is generated by $\alpha(L), \alpha(E_1), \ldots, \alpha(E_6).$ Hence, the element $\alpha \in \mathrm{Aut}\left(\mathrm{Pic}(S)_\mathbb{Q}\right)$ restricts to an automorphism of $\mathrm{Pic}(S)$ and, thus, lies in $W(\mathrm{E}_6).$

\end{proof}

Now we are going to prove two lemmas which are the main ingredients of the proof of Proposition~\ref{propZ/3Z}.

\begin{Lemma}\label{lemma:constructionofautomorphisms}
Let $S$ be a smooth cubic surface over a field $\mathbf{F}$ of characteristic different from $2$ and $3.$ Let $E_1, \ldots, E_6$ be pairwise skew lines on $S_{\overline{\mathbf{F}}}$ such that their union is Galois-invariant. Suppose that there is an element~$b \in W(\mathrm{E}_6)$  such that it fixes~$E_1$,~$E_2$ and~$E_3,$ cyclically permutes $E_4$, $E_5$ and~$E_6$ and commutes with the image of the Galois group $\mathrm{Gal}(\overline{\mathbf{F}}/\mathbf{F})$ in $W(\mathrm{E}_6).$  Then there is an element~\mbox{$r \in W(\mathrm{E}_6)$} such that it commutes  with the image of the Galois group in~$W(\mathrm{E}_6),$ and the elements~$r$ and~$b$  generate a group isomorphic to $(\mathbb{Z}/3\mathbb{Z})^2.$ Moreover, suppose that there is an element~$c$  in $W(\mathrm{E}_6)$  that  fixes~$E_4$,~$E_5$ and~$E_6$ and cyclically permutes~$E_1$,~$E_2$ and~$E_3.$ Then $r$ can be chosen in such a way that $r,$ $b$ and $c$ generate the group isomorphic to $(\mathbb{Z}/3\mathbb{Z})^3.$

\end{Lemma}

\begin{proof}

On~$S_{\overline{\mathbf{F}}}$ we denote by~$Q_i$ for all $i=1, \ldots, 6$ the pairwise skew lines such that 
$$
E_i \cdot Q_i=0 \quad \text{and} \quad E_i \cdot Q_j=1 \quad \text{for} \quad i \neq j.
$$

\noindent For $i,j \in \{1,\ldots, 6\}$ and $i<j$ we denote by $L_{ij}$ the remaining $15$ lines on $S_{\overline{\mathbf{F}}}$ such that 
\begin{gather*}
L_{ij} \cdot E_k=0 \quad \text{if} \quad k \notin \{i,j\} \quad \text{and} \quad L_{ij} \cdot E_k=1 \quad \text{if} \quad k \in \{i,j\};\\
L_{ij} \cdot Q_k=0 \quad \text{if} \quad k \notin \{i,j\} \quad \text{and} \quad L_{ij} \cdot Q_k=1 \quad \text{if} \quad k \in \{i,j\};\\
L_{ij} \cdot L_{kl}=0 \quad \text{if} \quad \{i,j\} \cap \{k,l\} \neq \varnothing \quad \text{and} \quad L_{ij} \cdot L_{kl}=1 \quad \text{if} \quad \{i,j\} \cap \{k,l\}=\varnothing.
\end{gather*}

Note that this description of the $(-1)$-curves on $S_{\overline{\mathbf{F}}}$ is equivalent the other description which says that $E_i$ for $i=1, \ldots, 6$ are the exceptional divisors of the blowup of $\mathbb{P}^2_{\overline{\mathbf{F}}}$ in $6$ points in general position $P_1, \ldots, P_6,$ respectively, $Q_i$ is the strict transform of the unique conic passing through all but one point $P_i$ among $P_1, \ldots, P_6$ for $i=1, \ldots, 6,$ and $L_{ij}$ is  the strict transform of the unique line passing through~$P_i$ and $P_j$ for $i,j=1, \ldots, 6$ and $i \neq j.$

Up to reordering, one can assume that $b$ acts as the permutation $(456),$ and $c$ as the permutation $(123)$ on $E_1, \ldots, E_6.$ By Proposition \ref{prop:WE6} there is an element~\mbox{$r \in W(\mathrm{E}_6)$} which fixes the canonical class $K_S$  and  maps the skew lines 
$$
E_1, \; E_2, \; E_3, \; E_4, \; E_5, \; E_6
$$
\noindent  to the other skew lines
$$
Q_1, \; Q_2, \; Q_3, \; L_{56}, \; L_{46}, \; L_{45},
$$
\noindent respectively.  Furthermore, one can see that~$r$ has order $3.$ Indeed, the element $r^2$ maps  the skew lines 
$$
E_1, \; E_2, \; E_3, \; E_4, \; E_5, \; E_6
$$
\noindent to other skew lines
$$
L_{23}, \; L_{13}, \; L_{12}, \; Q_4, \; Q_5, \; Q_6,
$$
\noindent respectively,  and $r^3(E_i)=E_i$ for all~$i \in \{1, \ldots, 6\}.$

 Let us   prove that the element $r$ commutes with the image $\Gamma$ of the Galois group~$\mathrm{Gal}(\overline{\mathbf{F}}/\mathbf{F})$ in $W(\mathrm{E}_6).$ It is enough to  check this on the curves~$E_1, \ldots, E_6.$  As $\Gamma$ commutes with $b,$ any elements of $\Gamma$ preserves the sets $\{E_1,E_2,E_3\}$ and~\mbox{$\{E_4,E_5,E_6\}.$} Let $g \in \Gamma.$ We will show that $gr(E_i)=rg(E_i)$ for $i=1,\ldots, 6.$ Consider $i \in \{1,2,3\}.$ Since $\Gamma$ preserves the set $\{E_1,E_2,E_3\},$ we have  $g(E_i)=E_l$ for some $l \in \{1,2,3\}.$ We obtain 
$$
gr(E_i)=g(Q_i)=Q_l=r(E_l).
$$ 
Consider $i \in \{4,5,6\}.$ Since $\Gamma$ preserves the set $\{E_4,E_5,E_6\},$ we have  $g(E_i)=E_l$ for some $l \in \{4,5,6\}.$ We obtain 
$$
gr(E_i)=g(L_{jk})=L_{mn}=r(E_l),
$$ 
\noindent where $\{i,j,k\}=\{l,m,n\}=\{4,5,6\}.$ Therefore, the element $r$ commutes with $\Gamma.$

For the elements $b$ and $r$ one has $br(K_S)=K_S=rb(K_S)$ and 
\begin{equation}\label{eq:br=rb}
\begin{gathered}
br(E_1)=Q_1=rb(E_1); \quad br(E_4)=L_{46}=rb(E_4);\\
br(E_2)=Q_2=rb(E_2);  \quad br(E_5)=L_{45}=rb(E_5);\\
br(E_3)=Q_3=rb(E_3);  \quad br(E_6)=L_{56}=rb(E_6).
\end{gathered}
\end{equation}
\noindent As $K_S, E_1, \ldots, E_6$ is a basis of $\mathrm{Pic}(S_{\overline{\mathbf{F}}}) \otimes \mathbb{Q}$, the elements~$b$ and $r$  commute.  Moreover, we obtain $r \neq b,b^2,$ because $b$ and $b^2$ fix the set~$E_1, \ldots, E_6,$ while $r$ does not. Thus, the elements $b$ and $r$ generate a group isomorphic to $(\mathbb{Z}/3\mathbb{Z})^2.$ 

Let $c \in W(\mathrm{E}_6)$ be an element  that   fixes~$E_4$,~$E_5$ and~$E_6$ and maps $E_1$ to~$E_2$,~$E_2$ to $E_3$ and~$E_3$ to $E_1.$ The elements $b$ and $c$ commute because they correspond to the elements~$(456)$ and $(123)$ in $\mathrm{S}_6 \subset W(\mathrm{E}_6),$ respectively. The elements $c$ and~$r$ also commute, which can be seen from a computation similar to \eqref{eq:br=rb}.  Hence, the elements $r$,~$b$ and $c$ generate a group isomorphic to~$(\mathbb{Z}/3\mathbb{Z})^n,$ where~\mbox{$n \leqslant 3.$} The elements~$b$ and~$c$ generate  a subgroup isomorphic to~\mbox{$H\simeq (\mathbb{Z}/3\mathbb{Z})^2$} which preserves the set~$\{E_1, \ldots, E_6\},$ while the element $r$ does not preserves this set. Thus, they generate a group~\mbox{$(\mathbb{Z}/3\mathbb{Z})^3 \subset W(\mathrm{E}_6).$ }

\end{proof}

\begin{Lemma}\label{lemma:constructionofautomorphismsz3z2}
Let $S$ be a smooth cubic surface over a field $\mathbf{F}$ of characteristic different from $2$ and $3.$ Let $E_1, \ldots, E_6$ be pairwise skew lines on $S_{\overline{\mathbf{F}}}$ such that their union be Galois-invariant. Suppose that there is an element~$b \in \mathrm{Aut}(S)$ of order~$3$ such that it fixes~$E_1$,~$E_2$ and $E_3$ and cyclically permutes $E_4$, $E_5$ and~$E_6.$ Then there is a biregular action  of~$(\mathbb{Z}/3\mathbb{Z})^2$ on $S,$ such that this group is generated by $b$ and some $r \in \mathrm{Aut}(S).$ Moreover,  assume that $c$ is an element in $\mathrm{Aut}(S)$ such that it  fixes~$E_4$,~$E_5$ and $E_6$ and cyclically permutes $E_1$, $E_2$ and~$E_3.$ Then $r$ can be chosen in such a way that~$b,$~$c$  and $r$  generate a subgroup isomorphic to $(\mathbb{Z}/3\mathbb{Z})^3$ in $\mathrm{Aut}(S).$

\end{Lemma}

\begin{proof}
Up to reordering, one can assume that $b$ acts as a permutation $(456),$ and $c$ as $(123)$ on $E_1, \ldots, E_6.$ Let us consider the element $r \in W(\mathrm{E}_6)$ from Lemma~\ref{lemma:constructionofautomorphisms}, so that $b$ and~$r$ generate the group $(\mathbb{Z}/3\mathbb{Z})^2$ in $W(\mathrm{E}_6),$ and $r$ commutes with the image of the Galois group $\mathrm{Gal}(\overline{\mathbf{F}}/\mathbf{F})$ in $W(\mathrm{E}_6).$ Moreover, if there is an element~\mbox{$c \in \mathrm{Aut}(S)$} such that  it  fixes~$E_4$,~$E_5$ and~$E_6$ and cyclically permutes $E_1$, $E_2$ and~$E_3,$ then the element $r$ can be chosen so that $b,$ $c$ and $r$ are generate the group $(\mathbb{Z}/3\mathbb{Z})^3$ in $W(\mathrm{E}_6).$

 It remains to check that~$r$ lies in $\mathrm{Aut}(S).$ The element $r$ commutes with the element~$b,$ and therefore,  by Lemma \ref{lemma:Sylow} it lies in the centralizer of $b$ in the automorphism group~\mbox{$\mathrm{Aut}(S_{\overline{\mathbf{F}}}) \subset W(\mathrm{E}_6).$} Moreover, since the element $r$ commutes with the image of the Galois group in~$W(\mathrm{E}_6),$   by Lemma~\ref{lemma:alphagal} it is contained in $\mathrm{Aut}(S).$

\end{proof}

\section{3-groups in the birational automorphism groups}\label{section:proofresults}
In this section we prove Proposition  \ref{propZ/3Z}. First of all, let us study the $3$-groups in the automorphism groups of del Pezzo surfaces and conic bundles.  Recall that a smooth projective  surface $S$ is a del Pezzo surface if its anticanonical class $-K_S$ is ample. The degree of a del Pezzo surface $S$ is defined as  $d=(K_S)^2.$ The following fact is well-known.

\begin{Lemma}[{see, for instance, \cite[Lemma 2.1]{Beauville}}]\label{lemma:klein}
Let $\mathbf{F}$ be an algebraically closed field of characteristic different from $3.$ Then $\mathrm{PGL}_2(\mathbf{F})$ does not contain $(\mathbb{Z}/3\mathbb{Z})^2.$
\end{Lemma}

Let us study $3$-subgroups in the  automorphism groups of del Pezzo surfaces of degree $d.$ Recall that $1 \leqslant d \leqslant 9.$

\begin{Lemma}\label{lemma:dP}
Let $\mathbf{F}$ be an algebraically closed field of characteristic different from~$3.$ Let $S$ be a del Pezzo surface of degree $d \neq 3$ over $\mathbf{F}.$ Then its automorphism group does not contain $(\mathbb{Z}/3\mathbb{Z})^3.$
\end{Lemma} 

\begin{proof}

Suppose that $d=9$, i.e. $S \simeq \mathbb{P}^2.$ Then $\mathrm{Aut}(S) \simeq \mathrm{PGL}_3(\mathbf{F}),$ and by Theorem~\ref{ThofBeauville}\ref{ThofBeauvilleaut3} we obtain $(\mathbb{Z}/3\mathbb{Z})^3 \not\subset \mathrm{Aut}(S).$

 Suppose that  $d=8$ and  $S \simeq \mathbb{P}^1 \times \mathbb{P}^1.$ We get that 
$$
\mathrm{Aut}(\mathbb{P}^1 \times \mathbb{P}^1) \simeq \left((\mathrm{PGL}_2(\mathbf{F}) \times \mathrm{PGL}_2(\mathbf{F})\right) \rtimes \mathbb{Z}/2\mathbb{Z}
$$
\noindent which does not contain $(\mathbb{Z}/3\mathbb{Z})^3$ because $\mathrm{PGL}_2(\mathbf{F})$ does not contain $(\mathbb{Z}/3\mathbb{Z})^2$ by Lemma \ref{lemma:klein}.

Suppose that either $d=8$ and $S \not\simeq \mathbb{P}^1 \times \mathbb{P}^1,$ or $d=7.$ Then we get an \mbox{$\mathrm{Aut}(S)$-equivariant} map $S \to \mathbb{P}^2.$ So we have $\mathrm{Aut}(S) \subset \mathrm{PGL}_3(\mathbf{F}).$  Thus, the group~$(\mathbb{Z}/3\mathbb{Z})^3$ is not contained in $\mathrm{Aut}(S).$

Suppose that $d=6.$ Then by~\cite[Theorem 8.4.2]{DolgachevAlg} we get 
$$
\mathrm{Aut}(S) \simeq (\mathbf{F}^*)^2 \rtimes \mathrm{D}_6.
$$
\noindent The subgroup $\mathrm{D}_6 \simeq \mathrm{S}_3 \times \mathbb{Z}/2\mathbb{Z}$ in $\mathrm{Aut}(S)$ is the dihedral group of order $12$ acting on the graph of $(-1)$-curves, which is a hexagon, and~\mbox{$(\mathbf{F}^*)^2 \subset \mathrm{Aut}(S)$} acts trivially on this graph.  If there is a subgroup $(\mathbb{Z}/3\mathbb{Z})^3 \subset \mathrm{Aut}(S)$ then the projection of this subgroup to $\mathrm{D}_6$ gives us either $\mathbb{Z}/3\mathbb{Z},$ or the trivial group. Hence there is a~$(\mathbb{Z}/3\mathbb{Z})^3$-invariant triple of  pairwise skew $(-1)$-curves. Thus, we can blow them down $(\mathbb{Z}/3\mathbb{Z})^3$-equivariantly and get $\mathbb{P}^2$ with the action of the group $(\mathbb{Z}/3\mathbb{Z})^3$ which is impossible.

 Suppose that  $d=5$ or $d=4.$ By~\cite[Corollary 8.2.40]{DolgachevAlg} the automorphism group~$\mathrm{Aut}(S)$ is contained in the Weyl group $W(\mathrm{A}_4) \; \text{or} \; W(\mathrm{D}_5),$  respectively.  The order of $W(\mathrm{A}_4)$ is equal to $120=2^3 \cdot 3 \cdot 5$ and the order of $W(\mathrm{D}_5)$ is equal to~$1920=2^7 \cdot 3 \cdot 5.$ Therefore, the automorphism group of~$S$ does not contain~$(\mathbb{Z}/3\mathbb{Z})^3.$

 Suppose that $d=2.$ Then the anticanonical linear system gives us a double cover
$$
\phi_{|-K_S|}: S \to \mathbb{P}^2.
$$  
\noindent Therefore, $\phi_{|-K_S|}$  induces the following exact sequence
$$
1 \to \mathbb{Z}/2\mathbb{Z} \to \mathrm{Aut}(S) \to \mathrm{Aut}(\mathbb{P}^2).
$$
\noindent Hence, $\mathrm{Aut}(S)$ does not contain $(\mathbb{Z}/3\mathbb{Z})^3.$

 Suppose that $d=1.$ Then the base locus of the linear system $|-K_S|$ is a point~$p.$  Hence, the point $p$ is fixed by $\mathrm{Aut}(S).$ Assume that $(\mathbb{Z}/3\mathbb{Z})^3 \subset \mathrm{Aut}(S).$ Therefore, the group $(\mathbb{Z}/3\mathbb{Z})^3$  acts faithfully on the Zariski tangent space $T_p(S)$ to $S$ at $p$. Thus, we obtain 
$$
(\mathbb{Z}/3\mathbb{Z})^3 \subset \mathrm{GL}\left(T_p(S)\right) \simeq \mathrm{GL}_2(\mathbf{F}),
$$
\noindent which is impossible.

\end{proof}

\begin{Lemma}\label{lemma:conicbundle}
Let $\mathbf{F}$ be a field of characteristic different from $3.$ Let $\phi:S \to B$ be a conic bundle over $\mathbf{F},$ where $B$ is a geometrically rational curve. Let $G$ be the subgroup in $\mathrm{Aut}(S)$  which consists of the elements mapping every fiber of $\phi$ to a fiber of $\phi$. Then $G$  does not contain $(\mathbb{Z}/3\mathbb{Z})^3.$
\end{Lemma} 

\begin{proof}
We have the following exact sequence
$$
1 \to \mathrm{Aut}_{\phi}(S) \to G \to \mathrm{Aut}(B),
$$
\noindent where $\mathrm{Aut}_{\phi}(S)$ is the group of all automorphisms of $S$ which map every fiber of~$\phi$ to itself. The group $\mathrm{Aut}_{\phi}(S)$ is contained in the automorphism group of the scheme-theoretic generic fiber $C_{\mathbf{F}(B)}$ of $\phi,$ where $C_{\mathbf{F}(B)}$ is a conic  over $\mathbf{F}(B)$ One has
$$
\mathrm{Aut}(C_{\eta}) \subset \mathrm{PGL}_2(\overline{\mathbf{F}(B)}).
$$

\noindent So by Lemma~\ref{lemma:klein} we get that neither $\mathrm{Aut}_{\phi}(S),$ nor $\mathrm{Aut}(B)$ contains $(\mathbb{Z}/3\mathbb{Z})^2.$ Therefore,  $G$ does not contain $(\mathbb{Z}/3\mathbb{Z})^3.$

\end{proof}

\begin{Cor}\label{cor:notomeganot3}
Let $\mathbf{F}$ be a perfect field of characteristic different from $3$ which does not contain non-trivial cube roots of unity. Let $V$ be a Severi--Brauer surface over~$\mathbf{F}$. Then 
$$
\mathrm{Bir}(V) \not\supset (\mathbb{Z}/3\mathbb{Z})^3.
$$

\end{Cor}

\begin{proof}
Assume that $G \simeq (\mathbb{Z}/3\mathbb{Z})^3$ is contained in $\mathrm{Bir}(V).$ Since $\mathbf{F}$ is perfect, the group $G$ acts biregularly either on a del Pezzo surface, or on a conic bundle~\mbox{$\phi:S \to B$} over a geometrically rational curve $B$ such that $\phi$ is equivariant with respect to $G$ (see, for example,~\cite[Lemma 3.6]{Chen} and~\cite[Lemma 14.1.1]{Prokhorov}). By Lemma~\ref{lemma:conicbundle} the latter case is impossible. By Lemma~\ref{lemma:dP} the group~$G$ does not act biregularly on a del Pezzo surface of degree~\mbox{$d \neq 3.$} Finally, by Lemma \ref{lemma:cubic^3}  the group~$G$ does not act biregularly on a del Pezzo surface of degree $d = 3.$

\end{proof}

Now we are ready to prove Proposition \ref{propZ/3Z}.

\begin{proof}[Proof of Proposition \ref{propZ/3Z}]
Let us prove \ref{eq:birsubset}. 
We can blow up  $V$ and get a smooth cubic surface $S$ such that the group $\mathbb{Z}/3\mathbb{Z}$ acts on $S$ fixing $3$ exceptional curves~$E_1,$~$E_2$ and~$E_3$ and cyclically permuting the other~$3$ exceptional curve~$E_4,$~$E_5$ and~$E_6.$ This follows from  Lemmas~\ref{lemma:autz/3z} and~\ref{lemma:p16} if $V$ is a non-trivial Severi--Brauer surface, and from Remark~\ref{remark:p16} for~$V \simeq \mathbb{P}^2.$ By Lemma \ref{lemma:constructionofautomorphismsz3z2} there is a biregular action of the group~$(\mathbb{Z}/3\mathbb{Z})^2$ on $S.$ This gives a birational action of $(\mathbb{Z}/3\mathbb{Z})^2$ on $V.$

Now let us prove \ref{eq:birnotsubset}. Assume that $\mathbf{F}$ contains a non-trivial cube root of unity. Then using  Lemmas~\ref{lemma:autz/3z} and~\ref{lemma:p16omega} if $V$ is a non-trivial Severi--Brauer surface, and  Remark~\ref{remark:p16omega} for~$V \simeq \mathbb{P}^2,$ we can blow up $V$ and get a smooth cubic surface $S$ such that the group~$(\mathbb{Z}/3\mathbb{Z})^2,$  with generators~$b$ and $c$ acts on $S$ as follows: the element~$b$ fixes~$3$ exceptional curves $E_1$, $E_2$ and $E_3$ and cyclically permutes the exceptional curves  $E_4$, $E_5$ and~$E_6,$ while the element $c$ fixes exceptional curves~$E_4$,~$E_5$ and~$E_6$ and cyclically permutes the exceptional curves~$E_1$,~$E_2$ and $E_3.$  By Lemma \ref{lemma:constructionofautomorphismsz3z2} there is a biregular action of the group~$(\mathbb{Z}/3\mathbb{Z})^3$ on $S.$ This gives a birational action of $(\mathbb{Z}/3\mathbb{Z})^3$ on $V.$ If $\mathbf{F}$ does not contain non-trivial cube roots of unity then by Corollary \ref{cor:notomeganot3} the group $(\mathbb{Z}/3\mathbb{Z})^3$ does not act birationally on $V.$ 

Assertion \ref{eq:birnotsubset4} follows from Theorem~\ref{ThofBeauville}\ref{ThofBeauvillebir4}.

Recall from Lemma~\ref{lemma:autz/3z} that the group~$\mathrm{Aut}(V)$ contains $\mathbb{Z}/3\mathbb{Z}.$ This gives~\ref{eq:autsubset}. If~$\mathbf{F}$ does not contain a non-trivial cube root of unity then by  Lemma~\ref{lemmaSB}  the group~$\mathrm{Aut}(V)$ does not contain~$(\mathbb{Z}/3\mathbb{Z})^2.$ Conversely,  if $\mathbf{F}$ contains a non-trivial cube root of unity then by Lemma~\ref{lemma:autz/3z} we get that $\mathrm{Aut}(V)$  contains $(\mathbb{Z}/3\mathbb{Z})^2.$ This gives~\ref{eq:autnotsubset}. Finally, from Theorem \ref{ThofBeauville}\ref{ThofBeauvilleaut3} we obtain \ref{eq:autnotsubset3}.

\end{proof}

\section{Proof of Theorem \ref{Z/3Z}}\label{section:proofoftheorem}
In this section we prove Theorem \ref{Z/3Z}.

\begin{proof}[Proof of Theorem \ref{Z/3Z}]
Let $G \subset \mathrm{Bir}(V)$ be a finite subgroup of birational automorphisms of $V.$ By Theorem \ref{ThofShramov} one has $G \subset (\mathbb{Z}/3\mathbb{Z})^3.$ By Proposition \ref{propZ/3Z}\ref{eq:birnotsubset} the group $G$ is contained in $(\mathbb{Z}/3\mathbb{Z})^2$ and by Proposition \ref{propZ/3Z}\ref{eq:birsubset} the group $(\mathbb{Z}/3\mathbb{Z})^2$ is contained in $\mathrm{Bir}(V)$.

Let $G \subset \mathrm{Aut}(V)$ be a finite subgroup of the automorphisms group of $V.$ By Theorem \ref{ThofShramov} one has $G \subset (\mathbb{Z}/3\mathbb{Z})^3.$ By Proposition \ref{propZ/3Z}\ref{eq:autnotsubset} the group $G$ is contained in $\mathbb{Z}/3\mathbb{Z}$ and by Proposition \ref{propZ/3Z}\ref{eq:autsubset} the group $\mathbb{Z}/3\mathbb{Z}$ is contained in $\mathrm{Aut}(V)$.

\end{proof}

\begin{Remark}
One of the facts used in the proof of~\cite[Theorem 1.3]{Shramov1} is the following assertion: the automorphism group of a Severi--Brauer surface $V$ over $\mathbb{Q}$ does not contain elements of prime order $p \geqslant 5.$ This fact follows, for instance, from~\cite[Theorem 6]{Serre}. Also, there are two alternative proofs of this fact provided in~\mbox{\cite[Lemma~7.1]{Shramov1}.} Unfortunately, the first of these proofs contains a gap: it treats identification~\mbox{$V_{\overline{\mathbb{Q}}} \simeq \mathbb{P}^2_{\overline{\mathbb{Q}}}$} as a Galois-invariant isomorphism, while it is obviously not Galois-invariant if $V$ is non-trivial. The same kind of gap is present in the proof of~\cite[Lemma 5.2]{Shramov1} (cf.~\cite[Lemma~5.3]{Shramov1}, where a similar trouble is luckily avoided). One can find a corrected proof of~\cite[Lemma~5.2]{Shramov1}  in  Example~\ref{exampleQF}.

\end{Remark}

%
%

\begin{thebibliography}{10}


\bibitem{Beauville}
A.~Beauville, \emph{$p$-Elementary subgroups of the Cremona group}, Journal of Algebra, \textbf{314} (2007), 553--564.



\bibitem{Frohlich}
J.~W.~S.~Cassels, A. Fr\"ohlich (eds.), \emph{Algebraic number theory}, Acad. Press (1968).

\bibitem{Carter}
R.~W.~Carter, \emph{Conjugacy classes in the Weyl group}, Compositio Math., \textbf{25} (1972), no. 1, 1--59.


\bibitem{Chatelet}
 F.~Ch\^atelet, \emph{Variations sur un th\`eme de H. Poincar\'e}, Ann. Sci. \'Ecole Norm. Sup.,  \textbf{61}
(1944), no. 3,  249--300.



\bibitem{Chen}
Y.~Chen, C.~A.~Shramov, \emph{Automorphisms of surfaces over fields of positive characteristic}, arXiv:2106.15906 [math.AG]


\bibitem{DolgachevAlg}
I.~V.~Dolgachev, \emph{Classical algebraic geometry. A modern view}, Cambridge university press, Cambridge (2012).

\bibitem{DolgachevDuncan}
I.~Dolgachev, A.~Duncan, \emph{Automorphisms of cubic surfaces in positive characteristic}, Izv. RAN. Ser. Mat., \textbf{83} (2019), no.3, 15--92.


\bibitem{DI}
I.~V.~Dolgachev, V.~A.~Iskovskikh, \emph{Finite Subgroups of the Plane Cremona Group}, In: Y.~Tschinkel, Y.~G.~Zarhin (eds), \emph{Algebra, Arithmetic, and Geometry. Progress in Mathematics}, \textbf{269} (2009), Birkh\"{a}user Boston.


\bibitem{Foote}
D.~S.~Dummit, R.~M.~Foote,  \emph{Abstract algebra}, Third edition. John Wiley \& Sons, Inc., Hoboken, NJ (2004). 


\bibitem{SimpleAlgebras}
Ph.~Gille, T.~Szamuely, \emph{Central Simple Algebras and Galois Cohomology}, Cambridge Studies in Advanced Mathematics, \textbf{165} (2017).




\bibitem{ShramovGorchinskiy}
S.~Gorchinskiy,  C.~Shramov, \emph{Unramified Brauer group and its applications}, Translations of Mathematical Monographs, \textbf{246}, American Mathematical Society, Providence, RI (2018).

\bibitem{IskTregub}
V.~A.~Iskovskikh, S.~L.~Tregub, \emph{Relations in the two-dimensional Cremona group over a nonclosed field},  Proc. Steklov Inst. Math., \textbf{207} (1995), 111--135.


\bibitem{Kollar}
J.~Koll\'{a}r,\emph{ Severi--Brauer varieties; a geometric treatment},   arXiv:1606.04368 [math.AG].
(2016).



\bibitem{Lang}
S.~Lang, \emph{Algebra}, Graduate Texts in Mathematics, Springer New York, NY (2002).




\bibitem{Manin}
Yu.~I.~Manin, \emph{Cubic forms: algebra, geometry, arithmetic}, In: North-Holland Mathematical Library, \textbf{4}, North-Holland Publishing Co., Amsterdam-London; American Elsevier Publishing Co., New York (1974).


\bibitem{Milnegroups}
J.~S.~Milne,  \emph{Algebraic groups. The theory of group schemes of finite type over a field}, Cambridge Studies in Advanced Mathematics, \textbf{170}, Cambridge University Press, Cambridge (2017).


\bibitem{Prokhorov}
Yu.~G.~Prokhorov, \emph{Equivariant minimal model program},  Russian Math. Surveys, \textbf{76}  (2021), no. 3 461--542.
 

\bibitem{RybakovTrepalin}
 S.~Yu.~Rybakov, A.~S.~Trepalin, \emph{Minimal cubic surfaces over finite fields}, Mat. Sb., \textbf{208} (2017),  no. 9, 148--170. 




\bibitem{Serre}
J.-P.~Serre, \emph{Bounds for the orders of the finite subgroups of $G(k)$}, Notes of lectures published in "Group Representation Theory", Lausanne, EPFL Press (2007).



\bibitem{Shramov1}
C.~A.~Shramov, \emph{Finite groups acting on Severi--Brauer surfaces}, European Journal of Mathematics, \textbf{7} (2021), 591--612.


\bibitem{Shramov3}
C.~A.~Shramov, \emph{Automorphisms of cubic surfaces without points}, International Journal of Mathematics \textbf{31} (2020), no. 11, 2050083.


\bibitem{ShramovVologodsky}
C.~A.~Shramov, V.~A.~Vologodsky, \emph{Boundedness for finite subgroups of linear algebraic groups}, Trans. Amer. Math. Soc. \textbf{374} (2021), 9029--9046.


\bibitem{Trepalin1}
A.~S.~Trepalin, \emph{Quotients of Severi--Brauer surfaces}, Dokl. RAN. Math. Inf. Proc. Upr., \textbf{501} (2021), 84--88.


\bibitem{Weinstein}
F.~W.~Weinstein, \emph{On birational automorphisms of Severi--Brauer surfaces}, Communications in Mathematics, \textbf{30} (2022), 1--9. 


\end{thebibliography}
%

\providecommand{\bysame}{\leavevmode\hbox to3em{\hrulefill}\thinspace}
\providecommand{\MR}{\relax\ifhmode\unskip\space\fi MR }
\providecommand{\MRhref}[2]{%
  \href{http://www.ams.org/mathscinet-getitem?mr=#1}{#2}
}
\providecommand{\href}[2]{#2}

\end{document}